\documentclass[a4paper]{article}
\AtBeginDocument{%
  \providecommand\BibTeX{{%
    \normalfont B\kern-0.5em{\scshape i\kern-0.25em b}\kern-0.8em\TeX}}}

\usepackage{enumitem}
\usepackage[utf8]{inputenc}
\usepackage{mathtools}
\usepackage{amsfonts}
\usepackage{amsthm}
\usepackage{amsmath}
\usepackage{amssymb}
\usepackage{tikz}
\usepackage{hyperref}
\usepackage{enumitem}

\newtheorem{Problem}{Problem}

\newtheorem{Theorem}{Theorem}
\newtheorem{Proposition}{Proposition}
\newtheorem{Lemma}{Lemma}

\newtheorem{Corollary}{Corollary}
\newtheorem{Algorithm}{Algorithm}
\newtheorem{Definition}{Definition}
\newtheorem{Example}{Example}

\newcommand{\uproman}[1]{\uppercase\expandafter{\romannumeral#1}}

\usepackage{pgfplots}\pgfplotsset{compat=1.13}
\usetikzlibrary{arrows,positioning,calc,patterns}

\def\clap#1{\hbox to0pt{\hss#1\hss}}
\def\eatspace#1{#1}
\def\step#1#2{\par\kern1pt\dimen44=#2em\advance\dimen44 1.67em\hangindent\dimen44\hangafter=1\noindent\rlap{\small#1}\kern\dimen44\relax\eatspace}
\let\set\mathbb
\def\<#1>{\langle#1\rangle}
\def\K{\set K}

\DeclareMathOperator{\sep}{sep}

\def\diag{\operatorname{diag}}
\def\val{\operatorname{val}}
\def\clap#1{\hbox to0pt{\hss#1\hss}}

\AtBeginDocument{%
  \providecommand\BibTeX{{%
    Bib\TeX}}}

\begin{document}

\title{Separating Variables in Bivariate Polynomial Ideals: the Local Case}

\date{}

\author{Manfred Buchacher}

\maketitle

\section{Abstract}
We present a semi-algorithm which for any irreducible $p\in\mathbb{K}[x,y]$ finds all elements of $\mathbb{K}(x) + \mathbb{K}(y)$ that are of the form $qp$ for some $q\in\mathbb{K}(x,y)$ whose denominator is not divisible by $p$.

\section{Introduction}

The following is the continuation of the work on an elimination problem that was started in~\cite{buchacher2020separating}. It discussed how $I \cap \left( \mathbb{K}[x] + \mathbb{K}[y]\right)$ can be determined when $I$ is an ideal of $\mathbb{K}[x,y]$. This article extends the ideas presented therein to ideals of the local ring of $\mathbb{K}[x,y]$ at an irreducible $p\in\mathbb{K}[x,y]$. It explains how to find all elements of $\mathbb{K}(x) + \mathbb{K}(y)$ that are of the form $qp$ for some $q\in\mathbb{K}(x,y)$ whose denominator is not a multiple of $p$. In contrast to~\cite{buchacher2020separating}, the procedure presented here is just a semi-algorithm. It terminates if $p$ has a non-trivial rational multiple in $\mathbb{K}(x) + \mathbb{K}(y)$, but it may not, when there is none. Termination depends on a dynamical system on the curve defined by $p$, and the finiteness of the orbit of a point thereon. 

The starting point of this work was a problem in enumerative combinatorics and the enumeration of restricted lattice walks. The question of how many there are can often be approached by considering their generating function and studying the functional equation they satisfy. A systematic study of equations that arise in this context, so-called discrete differential equations, was initiated in~\cite{bousquet2010walks, mishna2009classifying} and has received a lot of attention since then. We refer to~\cite{dreyfus2018nature, dreyfus2020walks, dreyfus2024enumeration} and the references therein for an overview of the relevant literature. In~\cite{bernardi2020counting, bousquet2023enumeration} it was explained how certain partial discrete differential equations can be reduced to ordinary discrete differential equations~\cite{bousquet2006polynomial}. The reduction relies on the existence of certain rational functions, so-called invariants and decoupling functions. Whether they exist, and in case they do, how to construct them, are therefore important questions. This work provides an answer on how to determine invariants. The construction of decoupling functions is discussed in~\cite{buchacher2024separated}. Motivated by the same line of research, (partial) answers on these questions can also be found in~\cite{bonnet2024galoisian, hardouin2021differentially}. 

Further applications are found in computer vision~\cite{binder2009algorithms}, the problem of parameter identification for ODE models~\cite{ovchinnikov2021computing}, and the design of diffractive optical systems~\cite{muller1998algorithmic},.

The problem this work is about has a field theoretic interpretation. By abuse of notation, let $\mathbb{K}(x,y)$ be the field that is generated by $x$ and $y$ over~$\mathbb{K}$, and assume that $p(x,y) = 0$ is the only relation satisfied by them. Let $\mathbb{K}(x)$ and $\mathbb{K}(y)$ be the subfields generated by $x$ and $y$, respectively. Their intersection $\mathbb{K}(x) \cap \mathbb{K}(y)$ might be different from $\mathbb{K}$, and one may ask what it is. Its computation amounts to finding all relations of the form $q(x,y) p(x,y) = f(x) - g(y)$, where $f$ and $g$ are rational functions that depend only on $x$ and $y$, respectively, and $q$ is a rational function, whose denominator is not a multiple of $p$.

 A semi-algorithm for computing intersections of intermediate fields of $\mathbb{K}$ and $\mathbb{K}(x_1,\dots,x_n)$ based on Gr{\"o}bner bases~\cite{buchberger2006bruno} is presented in~\cite[Algorithm~2.38]{binder2009algorithms}. It is shown~\cite[p. 37-38]{binder2009algorithms} to terminate when the fields are algebraically closed in $\mathbb{K}(x_1,\dots,x_n)$. It is shown in~\cite{ovchinnikov2021computing} that the condition can be relaxed. It is sufficient to require only one of the fields to be algebraically closed in the ambient rational function field~\cite[Proposition~27]{ovchinnikov2021computing}. An algorithm when the fields are linearly disjoint over their intersection is presented in~\cite{steinwandt2000freeness}. It follows a similar approach, again relying on Gr{\"o}bner bases computations. 

Intersections of algebraically closed fields are studied in~\cite{ash1986intersections} using methods from mathematical logic. An algorithm for their computation is sketched in~\cite[Theorem~2.3]{ash1986intersections}.


A characterization of when the intersection of two fields is different from their coefficient field is discussed in~\cite{fried1978poncelet}. It is related to the finiteness of correspondences between algebraic varieties~\cite[Theorem~1]{fried1978poncelet}. The arguments are geometric and Galois theoretic.

A semi-algorithm for computing $\mathbb{K}(x) \cap \mathbb{K}(y)$ given that $p(x,y) = 0$ is presented in~\cite{bonnet2024galoisian}. Similar as in~\cite{fried1978poncelet}, it characterizes the non-triviality of the intersection in terms of the finiteness of a correspondence. Furthermore, it explains how to construct a generator of the intersection when the correspondence is finite. It relies on Galois theoretic arguments and a constructive version of L{\"u}roth's theorem~\cite[Theorem~6.66]{rotman2015advanced}.

The particular problem of how to compute in the lattice of polynomial rational function fields is explained in~\cite{binder1996fast}. The approach is elementary. 

The present work can be seen as part of the more general problem of how to answer field theoretic questions effectively and efficiently~\cite{muller1999basic, muller2000grobner, kemper1993algorithm, sweedler1993using}. Some problems in this context have not yet found an algorithmic solution. And the solutions to others are computationally expensive. The present paper addresses the open problem of computing intersections of fields and proposes a linearization of the problem based on tropical~\cite{maclagan2015introduction} and Galois theoretic arguments. Though the procedure presented is just a semi-algorithm that solves the problem in its simplest form, we believe that the underlying ideas and arguments will be helpful to solve the problem in full generality. We also believe that they will be useful in answering other field theoretic questions such as, for instance, the field membership problem or more general decomposition problems~\cite{buchacher2024separated, bonnet2024galoisian}.

Polynomials and rational functions of the form $f(x) - g(y)$ have received considerable attention with regard to their reducibility~\cite{fried1973field, fried1987irreducibility,fried1974arithmetical, fried1973theorem, fried1980exposition, fried1999variables}, see also~\cite{geyer1994irreducibility}, or their solvability over the rational numbers, see~\cite{bilu2000diophantine}, the survey paper~\cite{kreso2015functional} and the references therein.

The paper is organized as follows. In Section~\ref{section:problem} we make precise what this article is about and give two different but equivalent formulations of the problem. In Section~\ref{section:hom} we explain how it can be solved for the particular case when $p$ is homogenous. And in Section~\ref{section:reduction} we show how the general case reduces to the homogenous one. The paper closes with Section~\ref{section:questions} and an open question on how the semi-algorithm could be turned into an algorithm.

This paper comes with an implementation in Mathematica. It can be found on https://github.com/buchacm/nearSeparation.git.

\section{Problem}\label{section:problem}

We assume throughout that $\mathbb{K}$ is an algebraically closed field of characteristic $0$. We denote by $\mathbb{K}[x,y]$ the ring of polynomials in $x$ and $y$ over $\mathbb{K}$, and we write $\mathbb{K}(x,y)$ for its quotient field. Given a rational function $r\in\mathbb{K}(x,y)$ in reduced form, we write $r_n$ and $r_d$ for its numerator and denominator, respectively. Conversely, given two coprime polynomials $r_n, r_d\in\mathbb{K}[x,y]$, we denote by $r$ their quotient $r_n/r_d$. Given $p\in\mathbb{K}[x]$, we denote by $\deg p$ its degree, and we write $\val p$ for its valuation, i.e. for the degree of the lowest order term of $p$.

\begin{Definition}
Let $p$ be an irreducible polynomial of $\mathbb{K}[x,y]\setminus\left( \mathbb{K}[x] \cup \mathbb{K}[y]\right)$. We write $\mathbb{K}[x,y]_p$ for the set of rational functions of $\mathbb{K}(x,y)$ whose denominator is not divisible by $p$. It is closed under the addition and multiplication of rational functions, and hence forms a ring. It is the \textbf{local ring} of $\mathbb{K}[x,y]$ at $p$.
\end{Definition}
The polynomial $p$ is an element of $\mathbb{K}[x,y]_p$. We denote the ideal it generates therein by $\langle p\rangle$. It consists of all rational functions of $\mathbb{K}(x,y)$ whose numerator is a multiple of $p$. An element of~$\mathbb{K}[x,y]_p$ has a multiplicative inverse if and only if it does not belong to $\langle p \rangle$. It is therefore the unique maximal ideal in $\mathbb{K}[x,y]_p$. Furthermore, $\mathbb{K}[x,y]_p$ is a principal ideal domain and every ideal is generated by some power of~$p$.

\begin{Problem}\label{prob:sep}
Given an irreducible polynomial $p\in\mathbb{K}[x,y]\setminus \left( \mathbb{K}[x] \cup \mathbb{K}[y]\right)$ and an ideal $I\subseteq \mathbb{K}[x,y]_p$, find a description of
\begin{equation*}
I \cap \left( \mathbb{K}(x) + \mathbb{K}(y)\right).
\end{equation*}
\end{Problem}

Although we have not made any restriction on the ideal in Problem~\ref{prob:sep}, it turns out that the problem is only interesting for $I = \langle p \rangle$.

\begin{Lemma}\label{lem:squareFree}
Let $p\in\mathbb{K}[x,y]\setminus\left(\mathbb{K}[x] \cup \mathbb{K}[y] \right)$ be irreducible, and let $m > 1$ be an integer. Then
\begin{equation*}
\langle p^m \rangle \cap \left( \mathbb{K}(x) + \mathbb{K}(y)\right) = \{0\}.
\end{equation*}
\end{Lemma}
\begin{proof}
Assume that there is a $q\in\mathbb{K}[x,y]_p\setminus\{0\}$ such that
\begin{equation*}
q p^m = f - g
\end{equation*}
for some $f\in\mathbb{K}(x)$ and $g\in\mathbb{K}(y)$. The numerator of the right hand side is $f_n g_d - g_nf_d$, and the numerator of the left hand side is $q_n p^m$. Hence
\begin{equation*}
p^m \mid f_n g_d - g_nf_d.
\end{equation*}
Since $p$ does not have any univariate factors, there is an $x_0\in\overline{\mathbb{K}(y)}\setminus\mathbb{K}$ that is a root of $f_n g_d - g_nf_d$ and $\frac{\partial}{\partial x}(f_n g_d - g_nf_d)$ in $\mathbb{K}(y)[x]$. It is then easily seen to be a root of $\frac{\partial}{\partial x} f$ too. Because $f$ is not a constant, $\frac{\partial}{\partial x} f$ cannot be identically zero. So $x_0\in\mathbb{K}$. A contradiction to the choice of $x_0$.
\end{proof}

There is an uncertainty in the formulation of Problem~\ref{prob:sep}. We asked for a ``description'' of the intersection of an ideal of $\mathbb{K}[x,y]_p$ with $\mathbb{K}(x) + \mathbb{K}(y)$ but we did not make clear what kind of description. In general, the intersection is not an ideal, so there is no point in asking for an ideal basis. And although it is a vector space, a vector space basis is not very helpful as it will be infinite in general.
The following two propositions provide an alternative description that will turn out to be convenient.

\begin{Proposition}
Let $p\in\mathbb{K}[x,y]\setminus\left( \mathbb{K}[x] \cup \mathbb{K}[y]\right)$ be irreducible. Then
\begin{equation*}
\mathrm{F}(p) := \left\{ (f,g) \in\mathbb{K}(x) \times \mathbb{K}(y) : f-g \in \langle p \rangle \right\}
\end{equation*}
is a field with respect to componentwise addition and multiplication. It is referred to as the \textbf{field of separated multiples} of $p$.
\end{Proposition}
\begin{proof}
Since $\mathrm{F}(p)$ is a subset of $\mathbb{K}(x)\times\mathbb{K}(y)$ and the latter is a commutative ring with respect to componentwise addition and multiplication it is enough to note that it contains $(0,0)$ and $(1,1)$ and to observe that $\mathrm{F}(p)$ is closed under componentwise addition and multiplication to prove that it is a ring with unity. It is clearly closed under componentwise addition, and it is closed under componentwise multiplication, because for $(f,g), (f',g')\in\mathrm{F}(p)$ we have $f-g, f'-g'\in\langle p \rangle$, and so $ff' - gg'=(f-g)f' + g(f'-g')\in\langle p\rangle$. Hence $\mathrm{F}(p)$ is indeed a ring. It is also a field since, if $(f,g)\in\mathrm{F}(p)$, then clearly $-(f,g)\in\mathrm{F}(p)$, and if $f\neq 0$, then also $g\neq0$, and $f^{-1} - g^{-1} = -f^{-1}g^{-1}(f-g)\in\langle p \rangle$ as $f$ and $g$ are units in $\mathbb{K}[x,y]_p$.
\end{proof}

\begin{Proposition}~\label{cor:simply}
Let $p\in\mathbb{K}[x,y]\setminus(\mathbb{K}[x]\cup \mathbb{K}[y])$ be irreducible. Then
\begin{equation*}
\mathrm{F}(p) = \mathbb{K}((f,g))
\end{equation*}
for some~$(f,g)\in\mathbb{K}(x) \times \mathbb{K}(y)$.
\end{Proposition}
\begin{proof}
The projection $\pi: \mathbb{K}(x)\times \mathbb{K}(y) \rightarrow \mathbb{K}(x)$ on the first component induces a field isomorphism between $\mathrm{F}(p)$ and its image $\pi(\mathrm{F}(p))$. By L\"{u}roth's theorem~\cite{luroth1875beweis} every subfield of $\mathbb{K}(x)$ that contains $\mathbb{K}$ is simple, i.e. of the form $\mathbb{K}(f)$ for some $f\in\mathbb{K}(x)$. Therefore $\mathrm{F}(p)$ is simple too.
\end{proof}

We can now formulate Problem~\ref{prob:sep} more precisely.

\begin{Problem}\label{prob:gen}
Given an irreducible polynomial $p\in\mathbb{K}[x,y]\setminus \left( \mathbb{K}[x] \cup \mathbb{K}[y] \right)$, find
a generator of $\mathrm{F}(p)$.
\end{Problem}

There is a formulation of Problem~\ref{prob:gen} which does not involve rational functions but only polynomials. It relies on the notion of near-separateness~\cite{alonso1997note, alonso1995rational} and near-separability.

\begin{Definition}
A polynomial $p\in\mathbb{K}[x,y]$ is said to be \textbf{near-separated}, if there exist $f\in\mathbb{K}(x)$ and $g\in\mathbb{K}(y)$ such that $p = f_n g_d - g_n f_d$. It is called \textbf{near-separable}, if there is a $q\in\mathbb{K}[x,y]\setminus \{0\}$ such that $qp$ is near-separated.
\end{Definition}

It is useful to be able to recognize whether a given polynomial is near-separated. In order to explain how, let $p$ be an element of $\mathbb{K}[x,y]$, and assume that it does not have any univariate factors. If $p$ is near-separated, then there are $f\in\mathbb{K}(x)$ and $g\in\mathbb{K}(y)$ such that $p = f_ng_d - g_nf_d$. Wlog we can assume that $\deg f_n > \deg f_d$ and $\mathrm{lc}(f_n) = 1$. So $\mathrm{lc}_x(p) = g_d$, and $g_d$ can be read off from $p$. The derivative of $p / \mathrm{lc}_y(p)$ with respect to $y$ equals $f_d \frac{\partial}{\partial_y} g$. Assuming that $\mathrm{lc}(f_d) = 1$, one can easily determine $f_d$ from that. Furthermore, $g_n$ can be reconstructed from it by making an ansatz, and solving a system of linear equations that results from a comparison of coefficients. An upper bound on the degree of $g_n$ is $\deg_y p$. The unknown $f_n$ is then easily derived from $p$.

The next lemma relates divisibility of near-separated polynomials with composition of rational functions. It will establish the existence of a distinguished near-separated multiple of a polynomial. For a proof we refer to~\cite[Theorem~1]{schicho1995note}.

\begin{Lemma}\label{lem:nearSep}
Let $ f, F\in\mathbb{K}(x)$ and $g, G\in\mathbb{K}(y)$ be non-constant rational functions. Then
\begin{equation*}
f_n g_d - g_n f_d  \mid F_nG_d - G_nF_d
\end{equation*}
if and only if 
\begin{equation*}
\exists \hspace{2 pt} h\in\mathbb{K}(t): h((f,g)) = (F,G).
\end{equation*}
\end{Lemma}

\begin{Corollary}\label{cor:minimal}
Any irreducible polynomial $p\in\mathbb{K}[x,y]$ has a near-separated multiple that divides any other near-separated multiple of $p$. It is unique up to multiplicative constants, and referred to as the \textbf{minimal near-separated multiple} of $p$. If $p\in\mathbb{K}[x,y]\setminus\left( \mathbb{K}[x] \cup \mathbb{K}[y]\right)$, then $\mathrm{F}(p) = \mathbb{K}((f,g))$ if and only if $f_ng_d - g_nf_d$ is the minimal near-separated multiple of $p$.
\end{Corollary}
\begin{proof}
The first part of the statement is clearly true when $p$ is near-separated. So let us assume that $p$ is not near-separated, and let $(F,G)$ be an element of $\mathbb{K}(x)\times\mathbb{K}(y)$ such that $F_nG_d - G_nF_d$ is a near-separated multiple of $p$. By Proposition~\ref{cor:simply} there is an $(f,g)\in\mathbb{K}(x)\times\mathbb{K}(y)$ such that $\mathrm{F}(p) = \mathbb{K}((f,g))$. Hence there is an $h\in\mathbb{K}(t)$ such that $(F,G) = h((f,g))$. By Lemma~\ref{lem:nearSep}, $f_ng_d - g_nf_d$ is a divisor of $F_nG_d - G_n F_d$. Since the latter was an arbitrary near-separated multiple of $p$, the former is a minimal near-separated multiple of $p$. Lemma~\ref{lem:nearSep} also shows that if $f_ng_d - g_nf_d$ is a minimal near-separated multiple of $p$, then $\mathrm{F}(p) = \mathbb{K}((f,g))$. Finally, minimal near-separated multiplies are unique. If there were another minimal near-separated multiple, they would divide each other, and therefore differ only by a multiplicative constant.
\end{proof}

We can now give the aforementioned reformulation of Problem~\ref{prob:gen}.
\begin{Problem}\label{prob:nearSep}
Given an irreducible polynomial $p\in\mathbb{K}[x,y]$, determine its minimal near-separated multiple.
\end{Problem}
We close this section with an example.
\begin{Example}\label{ex:1}
The polynomial $p=xy-x-y-x^2y^2$ is not near-separated as 
\begin{equation*}
\frac{\partial}{\partial_y}\frac{p}{\mathrm{lc}_x(p)} = -\frac{(1-x)}{y^2} -\frac{2x}{y^3}
\end{equation*}
is not the product of a polynomial in $x$ and a rational function in $y$. However, $p$ is near-separable. Its minimal near-separated multiple is
\begin{equation*}
(x-y) p = (1-x-x^3)y^2 - x^2(1-y-y^3),
\end{equation*}
and its field of separated multiples is
\begin{equation*}
\mathrm{F}(p) = \mathbb{K}\left(\left(\frac{1-x-x^3}{x^2},\frac{1-y-y^3}{y^2}\right)\right).
\end{equation*}
\end{Example}
 
\section{Homogenous case}\label{section:hom}

In this section we explain how to solve Problem~\ref{prob:gen} and Problem~\ref{prob:nearSep} when $p$ is homogenous. In order to do so, we first give some definitions. We introduce the notion of a weight function, explain what we mean by the leading part of a polynomial, and recall the definition of a homogenous polynomial.
\begin{Definition} A real-valued function $\omega$ on the set of terms in $x$ and $y$ is a \textbf{weight function}, if there are $\omega_x$, $\omega_y\in\mathbb{Z}$ such that
\begin{equation*}
\omega (a x^iy^j) = \omega_x i + \omega_y j
\end{equation*}
for all $i$, $j\in\mathbb{Z}$ and $a\in\mathbb{K}\setminus\{0\}$. In this case we write $\omega = (\omega_x, \omega_y)$. Two weight functions $\omega_1$, $\omega_2$ are said to be \textbf{equivalent}, if there is a positive number $c\in\mathbb{R}$ such that $\omega_2 = c\omega_1$. Given a polynomial $p\in\mathbb{K}[x,y]$ and a weight function $\omega$, we write $\omega(p)$ for the weight of a term of $p$ of maximal weight. We denote by $\mathrm{lp}_\omega(p)$ the sum of terms of $p$ of weight maximal weight. It is referred to as the \textbf{leading part} of $p$ with respect to $\omega$. It only depends on the equivalence class of the weight function but not on its representative.   We say that $p$ is \textbf{homogenous} with respect to $\omega$, if $\mathrm{lp}_\omega(p) = p$. 
Given a rational function $r\in\mathbb{K}(x,y)$, we define its leading part with respect to $\omega$ by $\mathrm{lp}_\omega(r) := \mathrm{lp}_\omega(r_n)/\mathrm{lp}_\omega(r_d)$. The \textbf{sign} of $\omega$ is 
\begin{equation*}
\mathrm{sgn}(\omega) := (\mathrm{sgn}(\omega_x), \mathrm{sgn}(\omega_y)),
\end{equation*}
where the sign of a number is either $1$, $-1$ or $0$ depending on whether it is positive, negative or equal to $0$. 
The \textbf{Newton polygon} of $p\in\mathbb{K}[x,y]$ is denoted by $\mathrm{Newt}(p)$. It is the convex hull of the support of $p$, that is, of the exponents of its non-zero terms.
\end{Definition}

Let $\omega = (\omega_x,\omega_y)$ be a non-zero weight function, and let $p$ be a polynomial that is homogenous with respect to it. We can assume that $\omega_x$ and $\omega_y$ are different from zero, otherwise $p$ is the product of a monomial in one and a polynomial in the other variable, and hence already near-separated. If there is a non-zero polynomial $q$ such that $qp$ is near-separated, we may assume that $q$, and therefore also $qp$, is homogenous. If it were not, we could replace $q$ by $\mathrm{lp}_\omega(q)$, since $\mathrm{lp}_\omega(q) p = \mathrm{lp}_\omega(q) \mathrm{lp}_\omega(p) = \mathrm{lp}_\omega(qp)$ and the leading part of a near-separated polynomial is near-separated. Under these assumptions there are $a,b\in\mathbb{K}$ and $k,l,m,n\in\mathbb{N}$ such that
\begin{equation*}
q p = a x^ky^l - b x^m y^n.
\end{equation*}
We can also assume that $qp$ is not a single term, as otherwise $p$ would be a single term as well. If $\omega_x$, $\omega_y  > 0$ and
$k>m$, for instance, then $n>l$, and $a x^ky^l - b x^m y^n$ is the product of $x^m y^l$ and $ax^{k-m} - by^{n-l}$ in $\mathbb{K}[x,y]$. Defining $\tilde{q}$ and $\tilde{p}$ by $q = x^{\mathrm{val}_x q} y^{\mathrm{val}_y q} \tilde{q}$ and $p = x^{\mathrm{val}_x p} y^{\mathrm{val}_y p} \tilde{p}$, we have
\begin{equation*}
\tilde{q} \tilde{p} =  ax^{k-m} - by^{n-l},
\end{equation*}
so $\tilde{p}$ has a non-trivial multiple in $\mathbb{K}[x] + \mathbb{K}[y]$. This is not only a necessary condition for the near-separability of $p$, but clearly also a sufficient one. We summarize these observations in the following proposition.
\begin{Proposition}
Let $p\in\mathbb{K}[x,y]$ be homogenous with respect to $\omega \in\mathbb{Z}_{> 0}^2$. Then $p$ is near-separable if and only if $x^{-\mathrm{val}_x p} y^{-\mathrm{val}_y p} p$ has a non-zero multiple in $\mathbb{K}[x]+\mathbb{K}[y]$.
\end{Proposition}
A similar statement holds in case not both of $\omega_x,\ \omega_y$ are positive. If $\omega_x,\ \omega_y < 0$, then one can replace them by their negative, and argue as before. If only one of them is negative, say $\omega_ y$, one needs to multiply $p$ by a suitable Laurent monomial and substitute $y^{-1}$ for $y$ before one can argue as before.

The question of how to decide whether a polynomial of $\mathbb{K}[x,y]$ has a non-zero multiple in $\mathbb{K}[x] + \mathbb{K}[y]$, and in case it does, how to find it, was discussed and solved in~\cite[Section~3]{buchacher2020separating}. If $q p = ax^m - b y^n$ for some $a,b\in\mathbb{K}\setminus\{0\}$, then $p(x,1)$ is a divisor of $ax^m - b$. Hence the roots of $p(x,1)$ are pairwise distinct and the ratio of every two of them is a root of unity. It turns out that this is also a sufficient condition for $p$ to have a non-zero multiple in $\mathbb{K}[x] + \mathbb{K}[y]$, and that a bound on the degrees of such a multiple can be derived from $p$. The precise statement, for whose proof we refer to~\cite{buchacher2020separating}, is the following.

\begin{Proposition}\label{prop:homog}
   Let $\omega$ be a weight function, and let $p \in \K[x, y] \setminus (\K[x] \cup \K[y])$ satisfy $\mathrm{lp}_{\omega}(p) = p$.
   Then $p$ has a non-zero multiple in $\mathbb{K}[x] + \mathbb{K}[y]$ if and only if
   \begin{enumerate}[label=(\alph*)]
       \item $p$ involves a monomial only in $x$, and
       \item all the roots of $p(x, 1)$ in $\K$ are distinct
       and the ratio of every two of them is a root of unity.
   \end{enumerate}
   Moreover, if $p$ has a non-zero multiple in $\mathbb{K}[x] + \mathbb{K}[y]$ and $N$ is the minimal number such that the ratio of every pair of roots of $p(x, 1)$ is an $N$-th root of unity, then the weight of the minimal (near-)separated multiple of $p$ is $N\omega_x$.
\end{Proposition}

It remains to clarify how to decide whether the quotient of every pair of roots of $p(x,1)$ is a root of unity. If $p(x,1)$ is monic and a divisor of $ax^n - b$, then its constant term equals $(b/a)^{\deg p(x,1) / n}$. We can therefore consider $p(x/c,1)$ for $c = (b/a)^{1/n}$ and check whether it is square-free and its roots are roots of unity. For the former it is sufficient to verify whether $p(x/c,1)$ and its derivative are co-prime. The latter can be done by computing the minimal polynomials of the roots of $p(x/c,1)$ over $\mathbb{K}$, and checking whether they are cyclotomic. If they are and if they are given by $\phi_{n_1}, \dots, \phi_{n_k}$, where $\phi_{n_i}$ is the $n_i$-th cyclotomic polynomial, that is, a divisor of $x^{n_i}-1$ but not of $x^d -1$ for $d< n_i$, then each root of $p(x/c,1)$ is an $N$-th root of unity for $N = \mathrm{lcm}(n_1,\dots,n_k)$.

\begin{Example}
Consider the polynomial $p = x^2y^2+xy+1$ which is homogenous with respect to $\omega = (1,-1)$. It is the product of $y^2$ and
\begin{equation*}
\tilde{p} = x^2 + x y^{-1} + y^{-2} \in\mathbb{K}[x,y^{-1}].
\end{equation*}
The latter has a non-zero multiple in $\mathbb{K}[x] + \mathbb{K}[y^{-1}]$, since $\tilde{p}(x,1)$ is the third cyclotomic polynomial. Its minimal (near-)separated multiple is
\begin{equation*}
(x-y^{-1}) \tilde{p} = x^3 - y^{-3}.
\end{equation*}
Consequently, $p$ is near-separable and its minimal near-separated multiple is
\begin{equation*}
y (x-y^{-1}) y^2 \tilde{p} = (xy-1) p = x^3y^3-1.
\end{equation*}
\end{Example}

\section{Reduction to the homogenous case}\label{section:reduction}

In this section we present a semi-algorithm that solves Problem~\ref{prob:gen} and Problem~\ref{prob:nearSep} and prove its correctness. We begin with presenting two necessary conditions for the near-separability of a polynomial.

\begin{Proposition}\label{lemma:necessary}
If $p\in\mathbb{K}[x,y]$ is near-separable, then so is its leading part $\mathrm{lp}_\omega(p)$ with respect to any weight function $\omega\in\mathbb{Z}^2$.
\end{Proposition}
\begin{proof}
Assume that $q\in\mathbb{K}[x,y]\setminus\{0\}$ is such that $qp$ is near-separated. Then $\mathrm{lp}_\omega(qp)$ is near-separated too, and $\mathrm{lp}_\omega(p)$ is near-separable, since $\mathrm{lp}_\omega(qp) = \mathrm{lp}_\omega(q)\mathrm{lp}_\omega(p)$.
\end{proof}

To compute the leading parts of a polynomial it is convenient to inspect its Newton polygon as there is a bijection between its leading parts and the faces of the polygon. The leading part which corresponds to a face is the sum of terms supported on it, that is, of those terms whose exponent vectors lie on it. Those we will be interested in consist of at least two terms. They correspond to the edges of the Newton polygon. The weight functions that give rise to them are the outward pointing normals of these edges. 

It can often be read off from the shape of its Newton polygon that a polynomial is not near-separable.

\begin{figure}
\begin{center}
  \begin{tikzpicture}[scale=.2]
    \draw (-3,-3)--(1,-3)--(1,3)--(-3,3);
    \draw (-3,-3)--(3,-3)--(3,1)--(-3,1);
    \draw[dotted] (3,1)--(1,3);
    \draw[->](-5,-3)--(6.5,-3);
    \draw[->](-3,-5)--(-3,6.5);
    \foreach \x/\y in {
    3/1, 1/3
   } \draw (\x,\y) node {$\bullet$};
  \end{tikzpicture}
  \hfil
   \begin{tikzpicture}[scale=.2]
    \draw (-3,-3)--(3,-3)--(3,3)--(-3,3);
    \draw (-3,-3)--(1,-3)--(1,1)--(-3,1);
    \draw[->](-5,-3)--(6.5,-3);
    \draw[->](-3,-5)--(-3,6.5);
    \foreach \x/\y in {
    3/3, 1/1
   } \draw (\x,\y) node {$\bullet$};
  \end{tikzpicture}
  \hfil
     \begin{tikzpicture}[scale=.2]
    \draw (-3,-3)--(3,-3)--(3,3)--(-3,3);
    \draw (-3,-3)--(1,-3)--(1,3)--(-3,3);
    \draw[->](-5,-3)--(6.5,-3);
    \draw[->](-3,-5)--(-3,6.5);
    \foreach \x/\y in {
    3/3, 1/3
   } \draw (\x,\y) node {$\bullet$};
  \end{tikzpicture}  
\end{center}
\end{figure}

\begin{Proposition}\label{prop:support}
Let $p\in\mathbb{K}[x,y]$ be near-separated, and let $\omega_1$, $\omega_2\in\mathbb{Z}^2$ be the outward-pointing normals of two distinct edges of its Newton polygon. Then $\mathrm{sign}(w_1)$ is different from $\mathrm{sign}(w_2)$.
\end{Proposition}
\begin{proof}
If $p\in \mathbb{K}[x] \cup \mathbb{K}[y]$, then the statement is clearly true. So let us assume that $p\in\mathbb{K}[x,y]\setminus\left( \mathbb{K}[x]\cup \mathbb{K}[y]\right)$, and let $f\in\mathbb{K}(x)$ and $g\in\mathbb{K}(y)$ be such that $p = f_ng_d - g_nf_d$. We will show that the Newton polygon of $p$ has at most one edge whose outward pointing normals have the sign vector $(1,1)$. For $(1,-1)$, $(-1,1)$ and $(-1,-1)$ the statement can be proven analogously after replacing $x$ by $x^{-1}$ and/or $y$ by $y^{-1}$ in $p$ and multiplying by suitable powers of $x$ and/or $y$. For $(1,0)$, $(0,1)$, $(-1,0)$ and $(0,-1)$ it is clearly true, since the Newton polygon of $p$ is convex.

Note that we can assume that $\deg f_n \neq \deg f_d$. If this were not the case we could perform a division with remainder on $f_n$ by $f_d$ and move the quotient, which is just a constant, to $g$. This does neither alter $f-g$, nor its numerator $p$. Wlog we assume that $\deg f_n < \deg f_d$. We can now distinguish three cases (compare with the figure above). If $\deg g_d > \deg g_n$, then the upper-right part of the Newton polygon of $p$ consists of (at most) three edges: possibly a horizontal edge, possibly a vertical one, and an edge whose outward pointing normal has only positive coordinates. They are the convex hulls of the supports of $\mathrm{lt}(g_d) f_n$ and $\mathrm{lt}(f_d) g_n$ and $\mathrm{lt}(f_n)\mathrm{lt}(g_d) - \mathrm{lt}(g_n)\mathrm{lt}(f_d)$. If $\deg g_d < \deg g_n$, then the upper-right part of the Newton polygon of $p$ consists of (at most) two edges, possibly a horizontal edge, and possibly a vertical edge, spanned by the supports of $\mathrm{lt}(g_n)f_d$ and $\mathrm{lt}(f_d)g_n$, respectively. If $\deg g_d = \deg g_n$, then the right part of the Newton polygon of $p$ is spanned by the support of $\mathrm{lt}(f_d) g_n$. By performing a division with remainder on $g_n$ by $g_d$ and moving the quotient to $f$ we can write $p = \tilde{f}_ng_d - \tilde{g}_nf_d$ where $\deg \tilde{f}_n = \deg f_d$ and $\deg g_d > \deg \tilde{g}_n$. The upper part of the Newton polygon of $p$ is therefore spanned by the support of $\mathrm{lt}(g_d)\tilde{f}_n$.
\end{proof}

\begin{Lemma}\label{lem:normals}
If $\omega\in\mathbb{R}^2$ is the outward pointing normal of an edge of the Newton polygon of $p\in\mathbb{K}[x,y]$, then it is the outward pointing normal of an edge of the Newton polygon of any non-zero multiple.
\end{Lemma}
\begin{proof}
A vector $\omega\in\mathbb{R}^2$ is an outward pointing normal for an edge of $\mathrm{Newt}(p)$ if and only if $\mathrm{lp}_\omega(p)$ involves at least two terms. If $q\in\mathbb{K}[x,y]\setminus\{0\}$ is such that $\omega$ is not an outward pointing normal for an edge of $\mathrm{Newt}(qp)$, then $\mathrm{lp}_\omega(qp)$ were a single term. But so would then also be $\mathrm{lp}_\omega(p)$, since it is a divisor of $\mathrm{lp}_\omega(qp)$.
\end{proof}

An immediate consequence of this lemma and the previous proposition is the following.

\begin{Corollary}
Let $\omega_1$, $\omega_2\in\mathbb{R}^2$ be the outward-pointing normals of two distinct edges of the Newton polygon of $p\in\mathbb{K}[x,y]$. If $p$ is near-separable, then $\mathrm{sign}(w_1)$ and $\mathrm{sign}(w_2)$ cannot be equal.
\end{Corollary}

\begin{Example}
The Newton polygon of $p = 1+ x^4+x^3y+y^2$ has two distinct edges whose outward pointing normals have the same sign vector. Hence $p$ is not near-separable.
\begin{figure}
\begin{center}
 \begin{tikzpicture}[scale= 1.0]
\fill[gray] (-3,-3)--(1,-3)--(0,-2)--(-3,-1)--cycle;
    \draw[-](-4,-3)--(2,-3);
    \draw[-](-3,-4)--(-3,0);
    \draw[->](-1.5,-1.5)--(-1.16667,-0.5);
    \draw[->](0.5,-2.5)--(1.20711,-1.79289);
    \foreach \x/\y in {
    -3/-3, -1/-3, 1/-3, 0/-2, -3/-1
} \draw (\x,\y) node {$\bullet$};
    \end{tikzpicture}
    \end{center}
    \begin{center}
\footnotesize{The Newton polygon of $1+x^2+x^4+x^3y+y^2$.}
\end{center}
\end{figure}

\end{Example}

Let $p\in\mathbb{K}[x,y]\setminus\left( \mathbb{K}[x] \cup \mathbb{K}[y] \right)$ be irreducible. We now turn to the question of how to compute a generator of $\mathrm{F}(p)$ when $p$ has a non-trivial rational multiple in $\mathbb{K}(x)+\mathbb{K}(y)$. 

Recall that a \textbf{pole} $s$ of $f\in\mathbb{K}(x)$ is either a finite or $\infty$. It is finite, if it is a root of $f_d$, and it is $\infty$, if $\deg f_n > \deg f_d$. If $s$ is a finite pole of $f$, then its \textbf{multiplicity} is its order as a root of $f_d$. If it is $\infty$, then its order $\deg f_n - \deg f_d$. In any case, its multiplicity is denoted by $\mathrm{m}(s,f)$.
 
 Assume that 
\begin{equation*}
qp = f - g
\end{equation*}
for (unknown) non-zero $q\in\mathbb{K}[x,y]_p$, $f\in\mathbb{K}(x)$ and $g\in\mathbb{K}(y)$. In order to find $f$, $g$ and $q$, it is sufficient to determine the poles of $f$ and $g$ and their multiplicities. For then we know the denominators of~$f$ and~$g$ and the degrees of their numerators, and by the following lemma, also the denominator of $q$ and the degree of its numerator.
\begin{Lemma}
Let $p\in\mathbb{K}[x,y]\setminus\left( \mathbb{K}[x] \cup \mathbb{K}[y]\right)$ be irreducible, and let $q\in\mathbb{K}[x,y]_p\setminus\{0\}$, $f\in\mathbb{K}(x)$ and $g\in\mathbb{K}(y)$ be such that $qp = f-g$. Then $q_d = f_d g_d$.
\end{Lemma}
\begin{proof}
The denominator of $qp$ is $q_d$ since $p$ and $q_d$ are relatively prime. And the denominator of $f-g$ is $f_dg_d$, since $f_n$ and $f_d$, and $g_n$ and $g_d$ are relatively prime. Hence $q_d = f_dg_d$.
\end{proof}

By making an ansatz for the numerators of~$q$ and $f$ and $g$, clearing denominators in $qp = f-g$, and comparing coefficients one finds a system of linear equations for the unknowns of the ansatz. Its non-trivial solutions give rise to non-zero rational functions~$f, g$ and $q$ such that $qp = f-g$.

In the next subsection we explain how to compute the poles of a generator of $\mathrm{F}(p)$, and in the subsection thereafter we show how to determine their multiplicities. The last subsection provides the arguments for the correctness of the resulting semi-algorithm.

\subsection{Poles}\label{sub:poles}

We will see that the poles of $f$ and $g$ appear in pairs and correspond to points on the projective curve associated with $p$. Their computation will define a dynamical system or algebraic correspondence thereon, the pairs of poles of $f$ and $g$ constituting an orbit thereof. The orbits are finite whenever $\mathrm{F}(p)$ is non-trivial, that is, isomorphic to $\mathbb{K}$. When $\mathrm{F}(p)$ is trivial, the orbit of a generic point is infinite. Nevertheless, the orbit of a specific point may be finite.    

Let assume that $\mathrm{F}(p)$ is non-trivial, and let $(f,g)$ be a generator for it. We may assume that $\infty$ is a pole of $f$. If $\deg f_n$ were smaller than $\deg f_d$, then we could replace $(f,g)$ by its reciprocal $(f^{-1},g^{-1})$, and if $\deg f_n$ and $\deg f_d$ were equal, then we could perform a division with remainder on $f_n$ by $f_d$, and move the quotient, which is just a constant, to $g$, and then consider the corresponding reciprocal.

Assuming that $\infty$ is a pole of $f$, we find that $\mathrm{lc}_x(p)$ is a divisor of $g_d$. Hence, each root of $\mathrm{lc}_x(p)$ is a pole of $g$. Furthermore, if $\deg \mathrm{lc}_x(p) < \deg_y p$, then $\infty$ is a pole of $g$. This holds, since if $\deg \mathrm{lc}_x(p) < \deg_y p$, then the Newton polygon of $p$ has an edge whose outward pointing normals have only positive coordinates. By Lemma~\ref{lem:normals}, this is also true for the Newton polygon of $f_ng_d - g_nf_d$. Hence $\deg g_n > \deg g_d$. If $s$ is a finite pole of $f$, then $p(s,y)$ is a divisor of $g_d$ and each root of $p(s,y)$ is a pole of $g$. If $s$ is a pole of $f$ and $\deg p(s,y) < \deg_y p(x+s,y)$, then an argument similar as before shows that $\infty$ is a pole of $g$. Of course, the situation is analogous with the roles of $f$ and $g$ switched.

Let $C$ be the projective curve in $\mathbb{P}^1(\mathbb{K})\times \mathbb{P}^1(\mathbb{K})$ associated with the bi-homogenization 
\begin{equation*}
p^{\mathrm{bihom}}([x_0:x_1],[y_0:y_1]) = x_0^{\deg_x p} y_0^{\deg_y p} p(x_1/x_0,y_1/y_0)
\end{equation*}
of $p$. The observations just made give rise to a procedure for computing the poles of $f$ and $g$ that can be interpreted as a dynamical system on $C$. It starts with those points of $C$ whose first coordinate is $\infty$, takes the horizontal lines through them and intersects them with the curve, then takes the vertical lines through these intersection points and determines their intersections, and continues in this way ad infinitum. The points constructed in this way constitute what we call the orbit of $\infty$.

\begin{Definition}
Let $\sim$ be the smallest equivalence relation on $C$ such that $(a_0,b_0) \sim (a_1,b_1)$ whenever $a_0 = a_1$ or $b_0 = b_1$. The equivalence class of $(a_0,b_0)$ is called the \textbf{orbit} of $(a_0,b_0)$. The orbit of any point whose first coordinate is $\infty$ will be referred to as the orbit at $\infty$, and for notational convenience just denoted by $\mathcal{O}$.
\end{Definition}

Under the assumption that $\mathrm{F}(p) \ncong \mathbb{K}$, there are only finitely many points that can be constructed in this way, because each point encountered in this process is a pair of poles for $f$ and $g$, and $f$ and $g$ have only finitely many poles. If $\mathrm{F}(p) \cong \mathbb{K}$, then there may be infinitely many such points, and then the procedure does not terminate. 

Let $\pi_i: \mathbb{P}^1(\mathbb{K})\times \mathbb{P}^1(\mathbb{K}) \rightarrow \mathbb{P}^1(\mathbb{K})$ denote the projection on the $i$-th coordinate.

\begin{Theorem}
Assume that $\mathrm{F}(p) \ncong \mathbb{K}$. Then there is a generator $(f,g)$ of $\mathrm{F}(p)$ such that the elements of $\pi_1(\mathcal{O})$ and $\pi_2(\mathcal{O})$ are among the poles of $f$ and $g$, respectively. In particular, $\mathcal{O}$ is finite.
\end{Theorem}

It is natural to ask whether $\mathcal{O}$ provides all the poles of a generator of $\mathrm{F}(p)$. We will later see that this is indeed the case, and thereby prove the correctness of the following algorithm.

Let $a\in\mathbb{P}^1(\mathbb{K})$, and define
\begin{equation*}
C(a) = C \cap \left(\{a\}\times \mathbb{P}^1(\mathbb{K})\right) \quad \text{ and } \quad {^{\mathrm{tr}}C}(a) = C \cap \left(\mathbb{P}^1(\mathbb{K}) \times \{a\}\right).
\end{equation*}

\begin{Algorithm}\label{alg:1}
  Input: an irreducible polynomial $p\in\K [x,y]\setminus \left( \mathbb{K}[x] \cup \mathbb{K}[y] \right)$ which has a non-trivial rational multiple in $\mathbb{K}(x) + \mathbb{K}(y)$.\\
  Output: a subset $S$ of $\left( \mathbb{K}\cup \{\infty\}\right)^2$ whose elements are points of the curve defined by $p$ such that $\pi_1(S)$ is the set of poles of some $f\in\mathbb{K}(x)$ and $\pi_2(S)$ is the set of poles of some $g\in\mathbb{K}(y)$ such that $\mathrm{F}(p)= \mathbb{K}((f,g))$.
  \step 10 Compute $S = {^{\mathrm{tr}}C}(\pi_2(C(\infty)))$, and
  \step 20 determine $\tilde{S} = {^{\mathrm{tr}}C}(\pi_2(C(\pi_1(S))))$.
  \step 30 While $\tilde{S} \neq S$, do:
  \step 41 set $S = \tilde{S}$, and
  \step 51 compute $\tilde{S} = {^{\mathrm{tr}}C}(\pi_2(C(\pi_1(S)))$.
  \step 60 Return $S$.
\end{Algorithm}

When $\mathrm{F}(p) \cong \mathbb{K}$, then the orbit of a generic point is infinite~\cite[Theorem~1]{fried1978poncelet}, and Algorithm~\ref{alg:1} may not terminate on input $p$, because $\mathcal{O}$ may not be finite. Proposition~\ref{lemma:necessary} and Proposition~\ref{prop:support} give two simple criteria for $\mathrm{F}(p)$ being trivial. But they do not provide a characterization thereof as~\cite[Theorem~3]{bousquet2010walks} and~\cite[Lemma~2.4]{hardouin2021differentially} show. Therefore Algorithm~\ref{alg:1} cannot be expected to terminate, when these criteria fail to apply. When $\mathcal{O}$, or any other orbit, is finite, we will later explain how a generator of $\mathrm{F}(p)$ can be computed. In particular, we will see how the non-triviality of $\mathrm{F}(p)$ can be decided. In general, however, this is still an open problem. There is some work on the particular case when $p$ is irreducible, of degree $2$ in both $x$ and $y$, and defines a curve of genus~$1$~\cite[Remark~5.1]{hardouin2021differentially}. In that case any orbit has either less than $24$ elements, or it is infinite. 

\subsection{Multiplicities}\label{sub:mult}

We have seen that the poles of $f$ and $g$ appear in pairs that correspond to points on the curve defined by $p$. In the following we will analyze the behavior of $p$ at these points. The guiding idea is that at a given point certain terms of $p$ contribute more than others. By discarding those which are negligible, a homogeneous polynomial remains. In Section~\ref{section:reduction} we explained how Problem~\ref{prob:sep} can be solved for homogeneous polynomials. For each point its solution will provide a $1$-parameter family for the multiplicities of the corresponding poles. These $1$-parameter families will merge to a single $1$-parameter family. We will then explain which parameter gives rise to the multiplicities of the poles of a generator of $\mathrm{F}(p)$.

Let us start with the following lemma.

\begin{Lemma}\label{lemma:weights}
Let $p \in\mathbb{K}[x,y]\setminus\left( \mathbb{K}[x] \cup \mathbb{K}[y] \right)$ be irreducible, and let $(s_1,s_2)\in\{0,\infty\}^2$ be a root (of the bi-homogenization) of $p$. Then there is an $\omega\in\mathbb{R}^2$ whose coordinates $\omega_x$ and $\omega_y$ are positive or negative depending on whether $s_1$ and $s_2$ equals $\infty$ or $0$, respectively, and $\mathrm{lp}_\omega(p)$ involves at least two terms.
\end{Lemma}

Assume that $qp = f - g$ and that $(s_1,s_2)\in\{0,\infty\}^2$ is a pair of poles of $f$ and $g$ that is a root of $p$. Furthermore, let $\omega\in\mathbb{Z}^2$ be a weight function as in Lemma~\ref{lemma:weights}. Since $\mathrm{lp}_\omega(p)$ involves at least two terms, so does $\mathrm{lp}_\omega(qp) = \mathrm{lp}_\omega(q)\mathrm{lp}_\omega(p)$. Therefore,
\begin{equation*}
\mathrm{lp}_\omega(q)\mathrm{lp}_\omega(p) = \mathrm{lp}_\omega(f) - \mathrm{lp}_\omega(g).
\end{equation*}
Hence, if $f_\omega\in\mathbb{K}[x^{\mathrm{sgn}(\omega_x)}]$ and $g_\omega\in\mathbb{K}[y^{\mathrm{sgn}(\omega_y)}]$ are such that
\begin{equation*}
\mathrm{F}(\mathrm{lp}_\omega(p)) = \mathbb{K}((f_\omega,g_\omega)),
\end{equation*}
then there is a positive integer $k \equiv k_{s_1,s_2}$ such that
\begin{equation*}
\mathrm{lp}_\omega(f) - \mathrm{lp}_\omega(g)  = f_\omega^k - g_\omega^k.
\end{equation*}
Up to sign, the degrees of $\mathrm{lp}_\omega(f)$ and $\mathrm{lp}_\omega(g)$ equal $\mathrm{m}(s_1,f)$ and $\mathrm{m}(s_2,g)$, respectively. So the above equation implies that
\begin{equation*}
(\mathrm{m}(s_1,f),\mathrm{m}(s_2,g)) = k \cdot \left(| \deg f_\omega | , | \deg g_\omega | \right).
\end{equation*}
We assumed that $s_1$, $s_2\in\{0,\infty\}$. If this is not the case, we can move the poles to $0$, and consider
\begin{equation*}
p_{s_1,s_2}(x,y) = p\left(x+s_1\cdot [s_1\in \mathbb{K}],y + s_2\cdot [s_2\in \mathbb{K}]\right),
\end{equation*}
where $[s_i\in \mathbb{K}]$ is $1$ or $0$ depending on whether $s_i$ is an element of $\mathbb{K}$ or not, and argue as before.
For each pair of poles determined by Algorithm~\ref{alg:1}, we can therefore compute their multiplicities up to a multiplicative constant $k$. It turns out that we can derive a system of linear equations for these constants from which all but one can be eliminated. If $(s_1,s_2)$ and $(t_1,t_2)$ are two pairs of poles of $(f,g)$ such that $s_1 = t_1$, then
 \begin{equation}\label{eq:mult}
 k_{s_1,s_2} \cdot {| \deg f_{\omega_{s_1,s_2}} |} = k_{t_1,t_2} \cdot {| \deg f_{\omega_{t_1,t_2}} |}.
 \end{equation}
An analogous equation holds when $s_2 = t_2$. So there is a linear relation between the unknowns $k_{s_1,s_2}$ and $k_{t_1,t_2}$ whenever $(s_1,s_2)$ and $(t_1,t_2)$ have a common component. The pairs of poles constituting an orbit implies that the solution space of these equations is at most $1$-dimensional. It is not $0$-dimensional since we assumed that $f$ and $g$ are not constants.

One could hope that each choice of parameter gives rise to the multiplicities of the poles of a multiple of $p$ in $\mathbb{K}(x) + \mathbb{K}(y)$. We will later see that this is indeed the case. We we will later also see which parameter gives rise to the multiplicities of the poles of a generator of $\mathrm{F}(p)$. For now we just formulate the algorithm these observations give rise to and illustrate it with an example.

\begin{Algorithm}\label{alg:2}
  Input: an irreduciblel $p\in\K [x,y]\setminus \left( \mathbb{K}[x] \cup \mathbb{K}[y] \right)$ with 
  $\mathrm{F}(p) \ncong \mathbb{K}$.\\
  Output: a generator $(f,g)$ of $\mathrm{F}(p)$.
  \step 10 Call Algorithm~\ref{alg:1} on $p$, and let $S$ be its output.
  \step 20 Set $M = \emptyset$, $E = \emptyset$.
  \step 30 For each $(s_1,s_2)\in S$ do:
  \step 41 compute a generator $(f_{\omega_{s_1,s_2}}, g_{\omega_{s_1,s_2}})$ of $\mathrm{F}(\mathrm{lp}_{\omega_{s_1,s_2}}(p))$ and enlarge $M$ by the pair consisting of $(s_1,s_2)$ and $k_{s_1,s_2} \cdot \left(| \deg f_{\omega_{s_1,s_2}} | , | \deg g_{\omega_{s_1,s_2}} |\right)$.
  \step 40 For any two elements $(s_1,s_2)$ and $(t_1,t_2)$ of $S$ do:
  \step 51 if $s_1 = t_1$, append equation~\eqref{eq:mult} to $E$, and if $s_2 = t_2$ replace $f_{\omega_{s_1,s_2}}$ and $f_{\omega_{t_1,t_2}}$ therein by $g_{\omega_{s_1,s_2}}$ and $g_{\omega_{t_1,t_2}}$, respectively, and append it to $E$.
  \step 60 Solve $E$ over $\mathbb{N}$, and substitute the generator of its solution space into $M$.
  \step 70 Determine $f_d$, $g_d$ and $q_d$, and make an ansatz for $f_n$, $g_n$ and $q_n$ according to the poles and multiplicities specified in $M$.
  \step 80 Equate the coefficients in $q_n p - f_ng_d + g_nf_d$ to zero and solve the resulting linear system for them.
  \step 90 Determine a non-trivial $(f,g)$ corresponding to a solution and return it.
\end{Algorithm}

\begin{figure}
\begin{center}
\begin{tikzpicture}[scale=2.]

 \draw [line width=0.25mm] plot coordinates {
(0.00651614, -0.975627) (0.00510973, -0.969395)
(0.00752896, -0.964286) (0.00892857, -0.956587) (0.00900228,
-0.955431) (0.00911774, -0.955168) (0.00580395, -0.946429)
(0.00988552, -0.938457) (0.0107257, -0.935703) (0.00765704,
-0.928571) (0.0109253, -0.92164) (0.0120867, -0.916485)
(0.00947163, -0.910714) (0.0120412, -0.904898) (0.0133433,
-0.897371) (0.0132767, -0.892857) (0.0131977, -0.888198)
(0.0145431, -0.878314) (0.0144621, -0.875) (0.0143768, -0.87152)
(0.0157065, -0.859293) (0.0156385, -0.857143) (0.0155687,
-0.854854) (0.0162234, -0.848214) (0.016844, -0.840299)
(0.0168062, -0.839286) (0.0167675, -0.838196) (0.0178571,
-0.822795) (0.0179162, -0.821488) (0.017915, -0.821429)
(0.0179365, -0.821349) (0.0185596, -0.804274) (0.0185352,
-0.803571) (0.0187693, -0.802659) (0.019224, -0.787081)
(0.0191602, -0.785714) (0.0195832, -0.783988) (0.0199062,
-0.769906) (0.0197905, -0.767857) (0.0203837, -0.765331)
(0.0206038, -0.752747) (0.0208811, -0.75) (0.0211751, -0.746682)
(0.0213154, -0.735601) (0.0216303, -0.732143) (0.0219602,
-0.72804) (0.0220067, -0.723214) (0.0220397, -0.718468)
(0.0217164, -0.714286) (0.0217645, -0.700336) (0.0223713,
-0.696429) (0.0230255, -0.69126) (0.0230333, -0.678571)
(0.0230051, -0.665862) (0.0237029, -0.660714) (0.0244425,
-0.654129) (0.0243806, -0.642857) (0.0258542, -0.63486) (0.025837,
-0.63298) (0.0250666, -0.625) (0.0266347, -0.616222) (0.0266309,
-0.615917) (0.0267857, -0.612504) (0.0269698, -0.607143)
(0.0272308, -0.598659) (0.027317, -0.597683) (0.026466, -0.589286)
(0.0277953, -0.581367) (0.0279796, -0.579163) (0.0271802,
-0.571429) (0.02703, -0.562744) (0.0279048, -0.553571) (0.0287927,
-0.542636) (0.0286401, -0.535714) (0.0284779, -0.528478)
(0.0293865, -0.517857) (0.0302952, -0.505419) (0.0301445, -0.5)
(0.0299848, -0.49427) (0.0309143, -0.482143) (0.0318287,
-0.468171) (0.0316962, -0.464286) (0.0315552, -0.460127)
(0.0324903, -0.446429) (0.033392, -0.430894) (0.0332964,
-0.428571) (0.0331936, -0.426051) (0.0341144, -0.410714)
(0.0349797, -0.393592) (0.0349434, -0.392857) (0.0350173,
-0.39216) (0.0357143, -0.376325) (0.0357616, -0.375) (0.0362713,
-0.3577) (0.0363537, -0.357143) (0.0364692, -0.356388) (0.0369537,
-0.340525) (0.0369605, -0.339286) (0.037179, -0.337821)
(0.0375721, -0.323286) (0.0375807, -0.321429) (0.0378973,
-0.319246) (0.0382031, -0.30606) (0.0382124, -0.303571) (0.038621,
-0.300665) (0.038844, -0.288844) (0.0388525, -0.285714)
(0.0393447, -0.282084) (0.0394905, -0.271633) (0.0394959,
-0.267857) (0.0400608, -0.263511) (0.0401361, -0.254422)
(0.0401353, -0.25) (0.0407381, -0.244976) (0.0407594, -0.232143)
(0.041422, -0.226435) (0.041382, -0.219953) (0.0413521, -0.214286)
(0.0420294, -0.207971) (0.0419468, -0.202661) (0.041889,
-0.196429) (0.0425498, -0.189593) (0.0424339, -0.185291)
(0.0423347, -0.178571) (0.0429401, -0.171346) (0.0427948,
-0.167795) (0.0426367, -0.160714) (0.0431406, -0.153288)
(0.0429547, -0.150098) (0.042717, -0.142857) (0.0433103,
-0.135261) (0.0424592, -0.125) (0.0412902, -0.112719) (0.041689,
-0.107143) (0.0419113, -0.100946) (0.0401444, -0.0892857)
(0.0374303, -0.0731446) (0.0374337, -0.0714286) (0.0373575,
-0.0697853) (0.0357143, -0.0636975) (0.0328026, -0.0535714)
(0.0292379, -0.0421907) (0.0256238, -0.0357143) (0.0178571,
-0.021942) (0.0162865, -0.0194278) (0.0150779, -0.0178571)
(0.00854416, -0.00931299) (4.71151*10^-15, -5.1556*10^-15)
(2.22045*10^-16, -2.22045*10^-16) (1.04083*10^-17,
  1.04083*10^-17) (-2.22045*10^-16, 2.22045*10^-16) (-5.1556*10^-15,
   4.71151*10^-15) (-0.00931299, 0.00854416) (-0.0178571,
  0.0150779) (-0.0194278, 0.0162865) (-0.0219304,
  0.0178571) (-0.0267857, 0.0208713) (-0.0305327,
  0.0230387) (-0.0334267, 0.0244982) (-0.0357143,
  0.0256923) (-0.0380014, 0.0267857) (-0.04252,
  0.0289086) (-0.0446429, 0.0296383) (-0.0488996,
  0.0310425) (-0.0535714, 0.0328026) (-0.0636975,
  0.0357143) (-0.0698776, 0.0372653) (-0.0714286,
  0.0374337) (-0.0731446, 0.0374303) (-0.0892857,
  0.0401444) (-0.100946, 0.0419113) (-0.107143,
  0.041689) (-0.112719, 0.0412902) (-0.125, 0.0424592) (-0.135261,
  0.0433103) (-0.142857, 0.042717) (-0.150098,
  0.0429547) (-0.153288, 0.0431406) (-0.160714,
  0.0426367) (-0.167795, 0.0427948) (-0.171346,
  0.0429401) (-0.178571, 0.0423347) (-0.184463,
  0.0416056) (-0.196429, 0.041889) (-0.202661,
  0.0419468) (-0.207971, 0.0420294) (-0.214286,
  0.0413521) (-0.219953, 0.041382) (-0.226435,
  0.041422) (-0.232143, 0.0407594) (-0.2372, 0.0407714) (-0.244956,
   0.0407585) (-0.25, 0.0401353) (-0.254422, 0.0401361) (-0.263511,
   0.0400608) (-0.267857, 0.0394959) (-0.271633,
  0.0394905) (-0.282084, 0.0393447) (-0.285714,
  0.0388525) (-0.288844, 0.038844) (-0.300665,
  0.038621) (-0.303571, 0.0382124) (-0.30606,
  0.0382031) (-0.319246, 0.0378973) (-0.321429,
  0.0375807) (-0.323059, 0.0373443) (-0.339286,
  0.0369605) (-0.340525, 0.0369537) (-0.356388,
  0.0364692) (-0.357143, 0.0363537) (-0.3577, 0.0362713) (-0.375,
  0.0357616) (-0.376325, 0.0357143) (-0.39216,
  0.0350173) (-0.392857, 0.0349434) (-0.393592,
  0.0349797) (-0.410714, 0.0341144) (-0.426051,
  0.0331936) (-0.428571, 0.0332964) (-0.430894,
  0.033392) (-0.446429, 0.0324903) (-0.460127,
  0.0315552) (-0.464286, 0.0316962) (-0.468171,
  0.0318287) (-0.482143, 0.0309143) (-0.49427, 0.0299848) (-0.5,
  0.0301445) (-0.505419, 0.0302952) (-0.517857,
  0.0293865) (-0.528478, 0.0284779) (-0.535714,
  0.0286401) (-0.542636, 0.0287927) (-0.553571,
  0.0279048) (-0.562744, 0.02703) (-0.571429,
  0.0271802) (-0.579163, 0.0279796) (-0.581367,
  0.0277953) (-0.589286, 0.026466) (-0.597683,
  0.027317) (-0.598659, 0.0272308) (-0.607143,
  0.0269698) (-0.612504, 0.0267857) (-0.615917,
  0.0266309) (-0.616222, 0.0266347) (-0.625, 0.0250666) (-0.63298,
  0.025837) (-0.63486, 0.0258542) (-0.642857,
  0.0243806) (-0.654129, 0.0244425) (-0.660714,
  0.0237029) (-0.665862, 0.0230051) (-0.678571,
  0.0230333) (-0.69126, 0.0230255) (-0.696429,
  0.0223713) (-0.700336, 0.0217645) (-0.714286,
  0.0217164) (-0.718468, 0.0220397) (-0.72804,
  0.0219602) (-0.732143, 0.0210682) (-0.735601,
  0.0213154) (-0.746682, 0.0211751) (-0.75, 0.0208811) (-0.752747,
  0.0206038) (-0.765331, 0.0203837) (-0.767857,
  0.0197905) (-0.769906, 0.0199062) (-0.783988,
  0.0195832) (-0.785714, 0.0191602) (-0.787081,
  0.019224) (-0.802659, 0.0187693) (-0.803571,
  0.0185352) (-0.804274, 0.0185596) (-0.821349,
  0.0179365) (-0.821429, 0.017915) (-0.821488,
  0.0179162) (-0.822795, 0.0178571) (-0.838196,
  0.0167675) (-0.839286, 0.0163695) (-0.840299,
  0.016844) (-0.854854, 0.0155687) (-0.857143,
  0.0156385) (-0.859293, 0.0157065) (-0.87152, 0.0143768) (-0.875,
  0.0144621) (-0.878314, 0.0145431) (-0.888198,
  0.0131977) (-0.892857, 0.0132767) (-0.897371,
  0.0133433) (-0.904898, 0.0120412) (-0.910714,
  0.00947163) (-0.916485, 0.0120867) (-0.92164,
  0.0109253) (-0.928571, 0.00765704) (-0.935703,
  0.0107257) (-0.938457, 0.00988552) (-0.946429,
  0.00580395) (-0.955168, 0.00911774) (-0.955431,
  0.00900228) (-0.956587, 0.00892857) (-0.964286,
  0.00752896) (-0.969395, 0.00510973)
(-0.975006, -0.00713674) (-0.970211, -0.00592538) (-0.964286,
-0.00887522) (-0.963986, -0.00892857) (-0.955815, -0.00938614)
(-0.954173, -0.0101125) (-0.946429, -0.00657594) (-0.9392,
-0.0106291) (-0.934293, -0.0121355) (-0.928571, -0.00906655)
(-0.922853, -0.0121389) (-0.914566, -0.0140058) (-0.910714,
-0.0117375) (-0.906701, -0.0138443) (-0.894847, -0.0158674)
(-0.892857, -0.0157968) (-0.890721, -0.0157213) (-0.883929,
-0.0167689) (-0.875083, -0.0177741) (-0.875, -0.01777) (-0.874909,
-0.0177658) (-0.874006, -0.0178571) (-0.858083, -0.0187974)
(-0.857143, -0.0187958) (-0.855762, -0.0192376) (-0.848214,
-0.0196165) (-0.841306, -0.0198778) (-0.839286, -0.0202486)
(-0.836471, -0.0206718) (-0.824625, -0.0210535) (-0.821429,
-0.0215607) (-0.817162, -0.0221233) (-0.808038, -0.0223239)
(-0.803571, -0.0220098) (-0.797819, -0.0236099) (-0.791549,
-0.0236916) (-0.785714, -0.0231871) (-0.778426, -0.0251458)
(-0.775162, -0.0251621) (-0.767857, -0.024426) (-0.758971,
-0.0267431) (-0.758886, -0.0267429) (-0.758427, -0.0267857)
(-0.75, -0.0273217) (-0.74204, -0.0277541) (-0.739783, -0.0280743)
(-0.732143, -0.0271147) (-0.725252, -0.0288236) (-0.720561,
-0.0294385) (-0.714286, -0.0285798) (-0.708545, -0.0299738)
(-0.701288, -0.0308549) (-0.696429, -0.031054) (-0.691925,
-0.0312107) (-0.681952, -0.0323338) (-0.678571, -0.0324502)
(-0.675399, -0.0325422) (-0.662543, -0.0338858) (-0.660714,
-0.0335778) (-0.658978, -0.0339779) (-0.643049, -0.0355223)
(-0.642857, -0.035488) (-0.642673, -0.0355299) (-0.640738,
-0.0357143) (-0.633929, -0.0361966) (-0.625985, -0.0366996)
(-0.625, -0.0367095) (-0.623709, -0.0370052) (-0.616071,
-0.0375031) (-0.609325, -0.0378961) (-0.607143, -0.0381901)
(-0.604321, -0.0385365) (-0.592763, -0.0391919) (-0.589286,
-0.0391772) (-0.584841, -0.0401593) (-0.576313, -0.0405991)
(-0.571429, -0.0405457) (-0.565255, -0.041888) (-0.55999,
-0.0421326) (-0.553571, -0.0420202) (-0.545547, -0.0437388)
(-0.543811, -0.0438105) (-0.535714, -0.0436157) (-0.526364,
-0.0450641) (-0.517857, -0.0453504) (-0.511053, -0.0467678)
(-0.506097, -0.0474745) (-0.5, -0.0472463) (-0.494754, -0.048325)
(-0.486203, -0.049511) (-0.482143, -0.0493304) (-0.478617,
-0.0500459) (-0.466125, -0.0517316) (-0.464286, -0.051636)
(-0.462673, -0.0519589) (-0.450386, -0.0535714) (-0.44676,
-0.0539024) (-0.446429, -0.053984) (-0.445903, -0.0540966)
(-0.428571, -0.0558449) (-0.414206, -0.0570635) (-0.410714,
-0.0579195) (-0.405189, -0.0590967) (-0.392857, -0.0602513)
(-0.384093, -0.0623357) (-0.383792, -0.0623635) (-0.382995,
-0.0625) (-0.375, -0.0636329) (-0.368105, -0.0645335) (-0.362971,
-0.0656009) (-0.357143, -0.065929) (-0.352697, -0.0669829)
(-0.341373, -0.0693416) (-0.339286, -0.0694486) (-0.337685,
-0.0698274) (-0.330017, -0.0714286) (-0.322732, -0.072732)
(-0.321429, -0.0729586) (-0.319227, -0.0736298) (-0.307812,
-0.0756695) (-0.303571, -0.0764154) (-0.296349, -0.0786513)
(-0.293453, -0.0791671) (-0.285714, -0.0805453) (-0.279066,
-0.0826379) (-0.272379, -0.084764) (-0.267857, -0.085579)
(-0.26505, -0.0864787) (-0.25648, -0.0892857) (-0.25146,
-0.0907454) (-0.25, -0.0912056) (-0.246707, -0.0925783)
(-0.238034, -0.095177) (-0.232143, -0.0971244) (-0.225154,
-0.100154) (-0.217855, -0.103573) (-0.214286, -0.104817)
(-0.212685, -0.105542) (-0.209463, -0.107143) (-0.200508,
-0.111222) (-0.196429, -0.113338) (-0.18909, -0.117662) (-0.17984,
-0.123732) (-0.178571, -0.12445) (-0.178235, -0.124664)
(-0.177766, -0.125) (-0.167634, -0.13192) (-0.160714, -0.137415)
(-0.157732, -0.139875) (-0.15455, -0.142857) (-0.148483,
-0.148483) (-0.142857, -0.15455) (-0.139875, -0.157732)
(-0.137415, -0.160714) (-0.13192, -0.167634) (-0.125, -0.177766)
(-0.124664, -0.178235) (-0.12445, -0.178571) (-0.123732, -0.17984)
(-0.117662, -0.18909) (-0.113338, -0.196429) (-0.111222,
-0.200508) (-0.107143, -0.209463) (-0.105542, -0.212685)
(-0.104817, -0.214286) (-0.103573, -0.217855) (-0.100154,
-0.225154) (-0.0971244, -0.232143) (-0.095177, -0.238034)
(-0.0925783, -0.246707) (-0.0912056, -0.25) (-0.0907454, -0.25146)
(-0.0892857, -0.25648) (-0.0864787, -0.26505) (-0.085579,
-0.267857) (-0.084764, -0.272379) (-0.0826379, -0.279066)
(-0.0805453, -0.285714) (-0.0791671, -0.293453) (-0.0786513,
-0.296349) (-0.0764154, -0.303571) (-0.0756695, -0.307812)
(-0.0736298, -0.319227) (-0.0729586, -0.321429) (-0.072732,
-0.322732) (-0.0714286, -0.330017) (-0.0698274, -0.337685)
(-0.0694486, -0.339286) (-0.0693416, -0.341373) (-0.0669829,
-0.352697) (-0.065929, -0.357143) (-0.0656009, -0.362971)
(-0.0645335, -0.368105) (-0.0636329, -0.375) (-0.0625, -0.382995)
(-0.0623635, -0.383792) (-0.0623357, -0.384093) (-0.0602513,
-0.392857) (-0.0590967, -0.405189) (-0.0579195, -0.410714)
(-0.0570635, -0.414206) (-0.0558449, -0.428571) (-0.0540966,
-0.445903) (-0.053984, -0.446429) (-0.0539024, -0.44676)
(-0.0535714, -0.450386) (-0.0519589, -0.462673) (-0.051636,
-0.464286) (-0.0517316, -0.466125) (-0.0500459, -0.478617)
(-0.0493304, -0.482143) (-0.049511, -0.486203) (-0.048325,
-0.494754) (-0.0472463, -0.5) (-0.0474745, -0.506097) (-0.0467678,
-0.511053) (-0.0453504, -0.517857) (-0.045595, -0.525834)
(-0.0453514, -0.527494) (-0.0436157, -0.535714) (-0.0438105,
-0.543811) (-0.0437388, -0.545547) (-0.0420202, -0.553571)
(-0.0421326, -0.55999) (-0.041888, -0.565255) (-0.0405457,
-0.571429) (-0.0405991, -0.576313) (-0.0401593, -0.584841)
(-0.0391772, -0.589286) (-0.0391919, -0.592763) (-0.0385365,
-0.604321) (-0.0381901, -0.607143) (-0.0378961, -0.609325)
(-0.0370052, -0.623709) (-0.0368391, -0.625) (-0.0366996,
-0.625985) (-0.0361966, -0.633929) (-0.0357143, -0.640738)
(-0.0355299, -0.642673) (-0.035488, -0.642857) (-0.0355223,
-0.643049) (-0.0339779, -0.658978) (-0.0335778, -0.660714)
(-0.0338858, -0.662543) (-0.0325422, -0.675399) (-0.0324502,
-0.678571) (-0.0323338, -0.681952) (-0.0312107, -0.691925)
(-0.0301377, -0.696429) (-0.0308549, -0.701288) (-0.0299738,
-0.708545) (-0.0297391, -0.714286) (-0.0294385, -0.720561)
(-0.0288236, -0.725252) (-0.0271147, -0.732143) (-0.0280743,
-0.739783) (-0.0277541, -0.74204) (-0.0257328, -0.75) (-0.0267429,
-0.758886) (-0.0267431, -0.758971) (-0.024426, -0.767857)
(-0.0251621, -0.775162) (-0.0251458, -0.778426) (-0.0231871,
-0.785714) (-0.0236916, -0.791549) (-0.0236099, -0.797819)
(-0.0220098, -0.803571) (-0.0223239, -0.808038) (-0.0221233,
-0.817162) (-0.0215607, -0.821429) (-0.0210535, -0.824625)
(-0.0206718, -0.836471) (-0.0198187, -0.839286) (-0.0198778,
-0.841306) (-0.0192376, -0.855762) (-0.0187958, -0.857143)
(-0.0187974, -0.858083) (-0.0178571, -0.874006) (-0.0177658,
-0.874909) (-0.01777, -0.875) (-0.0177741, -0.875083) (-0.0157213,
-0.890721) (-0.0157968, -0.892857) (-0.0158674, -0.894847)
(-0.0138443, -0.906701) (-0.0117375, -0.910714) (-0.0140058,
-0.914566) (-0.0121389, -0.922853) (-0.00906655, -0.928571)
(-0.0121355, -0.934293) (-0.0106291, -0.9392) (-0.00657594,
-0.946429) (-0.0101125, -0.954173) (-0.00938614, -0.955815)
(-0.00892857, -0.963986) (-0.00887522, -0.964286) (-0.00592538,
-0.970211) (-0.00713674, -0.975006) (-0.00420609, -0.982143)
};
   \draw  [line width=0.25mm] plot coordinates {
};
\draw [color = gray, line width = 0.25mm] (-1, 0) -- (0.75, 0);
\draw [color = gray, line width = 0.25mm] (0, -1) -- (0, 0.75);

\foreach \Point in {(-1,0),  (0,0), (0,-1)}{
    \node [color=red] at \Point {\textbullet};
}

\end{tikzpicture}
\end{center}
\begin{center}
\footnotesize{The curve defined by $xy-x-y-x^2y^2$ and the points on it defined by the dynamical system}
\end{center}
\end{figure}

\begin{Example}\label{ex:separable}
Let us again consider the polynomial $p = xy-x-y-x^2y^2$ from Example~\ref{ex:1}. We already know that it is near-separable, and we know how a generator of $\mathrm{F}(p)$ looks like. Still, let us compute this generator $(f,g)$ again, now using the ideas presented in this section. We first determine the poles of $f$ and $g$. They appear in pairs, and are points on the curve defined by $p$. Among them is $(\infty,0)$, since $\mathrm{lc}_x(p) = - y^2$, and $(0,0)$ and $(0,\infty)$ as $p(x,0) = -x$ and $\deg  p(0,y) < \deg_y p$, respectively. Since $p(0,y) = -y$, there are no further such pairs. See the figure above for a drawing of the curve and the pairs of poles on it. Next, we derive information on their multiplicities. For each pair $(s_1,s_2)$ just found, there is a weight function $\omega$ whose $i$-th component is positive or negative, depending on whether $s_i$ is $\infty$ or not, such that $\mathrm{lp}_{\omega}(p)$ consists of at least two terms. They are $(2,-1)$, $(-1,-1)$ and $(-1,2)$, and the corresponding leading parts are
\begin{equation*}
\mathrm{lp}_{(2,-1)}(p) = -x-x^2y^2, \quad \mathrm{lp}_{(-1,-1)}(p) = -x-y, \quad \mathrm{lp}_{(-1,2)}(p) = -y-x^2y^2.
\end{equation*}
For each of them, we solve Problem~\ref{prob:gen}. Since each leading part $\mathrm{lp}_\omega(p)$ is already near-separated, that is, of the form $f_ng_d - g_nf_d$, we find that $(f,g)$ is a generator of $\mathrm{F}(\mathrm{lp}_\omega(p))$. The fields of separated multiples for the leading parts are
\begin{equation*}
\mathbb{K}((x,-y^{-2})), \quad \mathbb{K}((x,-y)) \quad \text{and} \quad  \mathbb{K}((x^{-2},-y)).
\end{equation*}
Their generators show that there are $k_1,k_2,k_3\in\mathbb{N}$ such that
\begin{equation*}
\mathrm{m}(\infty,0) = k_1 \cdot (1,2), \quad
\mathrm{m}(0,0) = k_2 \cdot (1,1) \quad \text{and} \quad
\mathrm{m}(0,\infty) = k_3 \cdot (2,1).
\end{equation*}
The numbers $k_1,k_2$ and $k_3$ are not independent from each other, there are linear relations between them, since the second components of $\mathrm{m}(\infty,0)$ and $\mathrm{m}(0,0)$ and the first components of $\mathrm{m}(0,0)$ and $\mathrm{m}(0,\infty)$ are the same. We have $2k_1 = k_2$ and $2k_3 = k_2$. The solutions $(k_1,k_2,k_3)$ of these equations over $\mathbb{N}$ are positive multiples of $(1,2,1)$, and so there is a $k\in\mathbb{N}$ such that the multiplicities of $\infty$ and $0$ as poles of both $f$ and $g$ are $k$ and $2k$, respectively.
One could hope that $k = 1$ gives the multiplicities for a generator of $\mathrm{F}(p)$. Indeed, making the ansatz
\begin{equation*}
f = \frac{f_0 + f_1 x + f_2 x^2 + f_3x^3}{x^2} \quad \text{and} \quad g = \frac{g_0 + g_1 y + g_2y^2 + g_3y^3}{y^2}
\end{equation*}
and
\begin{equation*}
q = \frac{q_{00} + q_{10}x + q_{01}y}{x^2y^2},
\end{equation*}
clearing denominators in $qp = f - g$ and comparing coefficients, results in a system of linear equations for the undetermined coefficients whose solutions correspond to the rational functions
\begin{equation*}
f = \frac{u - u x - ux^3}{x^2} \quad \text{and} \quad g = \frac{u - u y - uy^3}{y^2}
\end{equation*}
and
\begin{equation*}
q = \frac{u x - u y}{x^2y^2},
\end{equation*}
for $u\in\mathbb{K}$. In particular, we find that
\begin{equation*}
\frac{x-y}{x^2y^2} p = \frac{1-x-x^3}{x^2} - \frac{1-y-y^3}{y^2}.
\end{equation*}
\end{Example}

We finish this section with another example.

\begin{Example}
The polynomial $p = x^2 + 3xy + y^2$ is not near-separable. Still, the semi-algorithm terminates on input $p$ as the only point determined by Algorithm~\ref{alg:1} is $(\infty,\infty)$. It is a singularity of the curve associated with (the bi-homogenization) of $p$.
\end{Example}

\subsection{Correctness}

In this subsection we prove the correctness of Algorithm~\ref{alg:2}. We show that it does not miss any pole of a generator of $\mathrm{F}(p)$ and that it does not fail in determining their multiplicities. Our arguments generalize the arguments in~\cite[Section~3]{buchacher2020separating}.

Let $f\in\mathbb{K}(x)$ and $g\in\mathbb{K}(y)$ be non-constant rational functions. We study the near-separated polynomial
\begin{equation*}
f_n g_d - g_nf_d
\end{equation*}
by introducing a new variable $t$ and investigating the auxiliary equations
\begin{equation*}
f = t \quad \text{and} \quad g = t.
\end{equation*}
We solve these equations with respect to $x$ and $y$ in $\overline{\mathbb{K}(t)}$, the algebraic closure of $\mathbb{K}(t)$. Let their solutions be $\alpha_0, \dots, \alpha_{m-1}$ and $\beta_0,\dots,\beta_{n-1}$, respectively, where $m = \max\{\deg f_n, \deg f_d\}$ and $n = \max\{\deg g_n, \deg g_d\}$. We will throughout view $\overline{\mathbb{K}(t)}$ as a subfield of $\mathbb{K}\{\{t^{-1}\}\}$, the field of Puiseux series in descending powers of $t$. The $\alpha_i$'s and $\beta_j$'s are therefore of the form
\begin{equation*}
c_1 t^{d_1} + c_2 t^{d_2} + \dots,
\end{equation*}
where $c_i\in\mathbb{K}$ and $d_1 > d_2 > \dots$ are rational numbers which have a common denominator. The construction underlying the Newton-Puiseux algorithm~\cite{walker1950algebraic} shows that their leading terms encode the poles of $f$ and $g$ as well as their multiplicities in the following sense.
\begin{Proposition}\label{prop:newt}
Let $f\in\mathbb{K}(x)$, and let $s\in\mathbb{K}\cup \{\infty\}$ be a pole of multiplicity $m$. 
If $s = \infty$, then for each root $c$ of $\mathrm{lc}(f_d) + \mathrm{lc}(f_n) t^m$ there is a root of $f-t$ in $\mathbb{K}\{\{t^{-1}\}\}$ whose leading term is $c t^{1/m}$.
If $s = 0$, then for each root $c$ of $\mathrm{lc}(f_n(x^{-1})) + \mathrm{lc}(f_d(x^{-1}))t^m$ there is a series root of $f-t$ whose leading term is $c t^{-1 / m}$. 
And if $s$ is neither $0$ nor $\infty$, then there are $m$ series roots whose leading term is $s$. We say that such series are associated with $s$.
\end{Proposition}

Every element $\pi$ of $\mathrm{Gal}(\overline{\mathbb{K}(t)}/\mathbb{K}(t))$, the Galois group of $\overline{\mathbb{K}(t)}$ over $\mathbb{K}(t)$, acts on $\mathbb{Z}_m\times \mathbb{Z}_n$ by
\begin{equation*}
\pi(i,j):= (i',j') \quad :\Longleftrightarrow \quad  (\pi(\alpha_i),\pi(\beta_j))=(\alpha_{i'},\beta_{j'}).
\end{equation*}
Let $G\subseteq \mathrm{S}_m\times \mathrm{S}_n$ be the group of permutations induced on $\mathbb{Z}_m\times\mathbb{Z}_n$ by this action. In the following we study subsets $T\subseteq\mathbb{Z}_m\times \mathbb{Z}_n$ that are invariant under the action of $G$ and investigate how they relate to factors of $f_n g_d - g_nf_d$.

For a subset $T\subseteq \mathbb{Z}_m\times \mathbb{Z}_n$, and $(i,j)\in\mathbb{Z}_m\times\mathbb{Z}_n$, we introduce
\begin{equation*}
T_{i,*} = \{k \mid (i,k)\in T\} \quad \text{and} \quad  T_{*,j} = \{k \mid (k,j)\in T\}.
\end{equation*}
As in~\cite{buchacher2020separating}, we have the following two lemmas.
\begin{Lemma}\label{lem:card}
Let $T\subseteq\mathbb{Z}_m\times \mathbb{Z}_n$ be invariant under the action of $G$. Then
\begin{equation*}
|T_{0,*}| = |T_{1,*}| = \dots = |T_{m-1,*}| \quad \text{and} \quad  |T_{*,0}| = |T_{*,1}| = \dots = |T_{*,n-1}|.
\end{equation*}
\end{Lemma}
\begin{proof}
We only show that $|T_{0, \ast}| = |T_{1, \ast}|$. The other equalities are shown analogously.
Observe that $f_n - t f_d$ is irreducible over $\K(t)$.
  If it were not, it would be reducible over~$\K[t]$ due to Gauss's lemma.
This, however, is impossible, because $f_n - tf_d$ is linear in~$t$, and does not have any non-trivial factors in $\K[x]$ since $f_n$ and $f_d$ are relative prime.
  The irreducibility of $f_n - t f_d$ implies that its Galois group acts transitively on its roots.
  In particular, there exists $\pi\in\mathrm{Gal}(\overline{\mathbb{K}(t)}/\mathbb{K}(t))$ such that $\pi(\alpha_0) = \alpha_1$.
  Hence $\pi$ maps $T_{0, \ast}$ to $T_{1, \ast}$, and we have~$|T_{0, \ast}| \leqslant |T_{1, \ast}|$.
  The reverse inequality is proven analogously.

\end{proof}

\begin{Lemma}\label{lemma:factor}
The map
\begin{equation*}
p \quad \mapsto \quad T:= \{(i,j)\in\mathbb{Z}_m\times \mathbb{Z}_n : p(\alpha_i,\beta_j) = 0\}
\end{equation*}
defines a bijection between the set of factors of $f_ng_d - g_nf_d$ (up to multiplicative constants) and the set of subsets of $\mathbb{Z}_m\times\mathbb{Z}_n$ that are invariant under the action of $G$.
\end{Lemma}
\begin{proof}
Let $p$ be a divisor of $f_ng_d - g_nf_d$, and let $T$ be the corresponding subset of $\mathbb{Z}_m\times\mathbb{Z}_n$. If $(i,j)\in T$, then $p(\alpha_i,\beta_j) = 0$, and so $p(\pi(\alpha_i),\pi(\beta_j)) = 0$ for any $\pi\in\mathrm{Gal}(\overline{\mathbb{K}(t)}/\mathbb{K}(t))$. Therefore, $\pi(i,j) \in T$, and $T$ is $G$-invariant.

Let $T$ be a $G$-invariant subset of $\mathbb{Z}_m\times\mathbb{Z}_n$, and let~$T_{0,*} = \{j_1,\dots,j_s\}$. Since $f(\alpha_0)=t$, we have $\mathbb{K}(\alpha_0)\supseteq \mathbb{K}(t)$, so $T$ is invariant with respect to the action of the Galois group $\mathrm{Gal}(\overline{\mathbb{K}(t)}/\mathbb{K}(\alpha_0))$. If $\alpha_0$ is fixed, then $\beta_{j_1},\dots,\beta_{j_s}$ are permuted. Therefore, $(y-\beta_{j_1})(y-\beta_{j_2})\dots(y-\beta_{j_s})$ is invariant under the action of $\mathrm{Gal}(\overline{\mathbb{K}(t)}/\mathbb{K}(\alpha_0))$. Hence, by the fundamental theorem of Galois theory, it is a polynomial in~$\mathbb{K}(\alpha_0)[y]$.
By construction, its numerator divides $f_n(\alpha_0)g_d(y) - g_n(y)f_d(\alpha_0)$ in $\mathbb{K}[\alpha_0,y]$. Replacing $\alpha_0$ by~$x$ therefore results in a polynomial $p\in\mathbb{K}[x,y]$ that divides~$f_ng_d - g_nf_d$ in $\mathbb{K}[x,y]$.

  It remains to show that the two constructions are inverse to each other. We first prove that the invariant set associated with the polynomial $p$ just constructed equals $T$. Let $(i, j) \in \mathbb{Z}_m \times \mathbb{Z}_n$.
  Since $\mathrm{Gal}(\overline{\mathbb{K}(t)}/\mathbb{K}(t))$ acts transitively on the roots of $f - t$, there is an automorphism $\pi$
  with $\pi (\alpha_i) = \alpha_0$.
  Let $\beta_{j'} = \pi(\beta_j)$.
  We then have
  \[
    p(\alpha_i, \beta_j) = 0 \iff p(\alpha_0, \beta_{j'}) = 0 \iff j' \in T_{0, \ast} \iff (i, j) \in T.
  \]
 The first of these equivalences follows from $\pi$ being an automorphism, the second from the construction of $p$, and the third from the invariance of $T$.
We now show that $p$ is the unique factor of $f_ng_d - g_nf_d$ whose associated invariant set is $T$. Assume that $\tilde{p}$ is another divisor such that $\tilde{p}(\alpha_i,\beta_j) = 0$ if and only if $(i,j)\in T$.
  The same argument which proved that the polynomial constructed from $T$ is a divisor of $f_ng_d - g_nf_d$ applies to show that $p$ is a divisor of $\tilde{p}$ in $\K [x,y]$, and vice versa. Hence they only differ by a multiplicative constant.
\end{proof}

\begin{Example}\label{ex:invSet}
There are four invariant subsets of $\mathbb{Z}_4\times \mathbb{Z}_4$ that can be associated with
\begin{equation*}
f_ng_d - g_nf_d = (1-x)^2(1+x+x^2)y(1+y)^2 + (1+y+y^2)^2x^2.
\end{equation*}
The diagram on the left below illustrates the invariant set $T$ that corresponds to $p = xy -1-y-x^2y-x^2y^2$. Its rows are numbered by the roots of $f - t$, and the roots of $g - t$ number its columns. A dot in the $i$-th row and $j$-th column indicates that $p$ annihilates $(\alpha_i,\beta_j)$. The other invariant sets are $\emptyset$, $T^c$ and $\mathbb{Z}_m\times\mathbb{Z}_n$. The first and last correspond to the trivial factors $1$ and $f_ng_d - g_nf_d$, the second one is associated with the complementary factor of $p$.
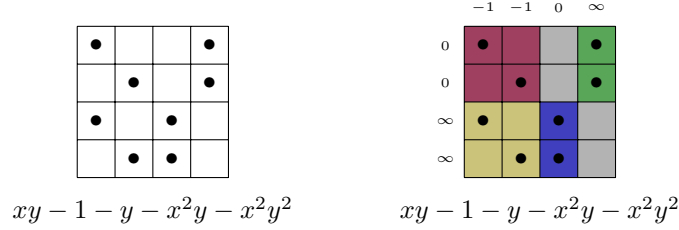
\begin{center}
\begin{tikzpicture}[scale=0.5]
  \draw[xshift=-.5cm,yshift=-.5cm] (0,0) grid (4,4);
  \foreach\x/\y in {1/0, 2/0, 0/1, 2/1, 1/2, 3/2, 0/3, 3/3} \draw (\x,\y) node {$\bullet$}; 
    \draw (1.5,-.75) node[below] {\vbox{\clap{$\mathstrut xy -1-y-x^2y-x^2y^2$}\kern0pt}};
\end{tikzpicture}
\hfil
\begin{tikzpicture}[scale=0.5]
 \fill[yellow!50!gray] (-0.5,-0.5) rectangle (1.5,1.5);
  \fill[purple!50!gray] (-0.5,1.5) rectangle (1.5,3.5);
  \fill[blue!50!gray] (1.5,-0.5) rectangle (2.5,1.5);
  \fill[black!30!white] (1.5,1.5) rectangle (2.5,3.5);
  \fill[black!30!white] (2.5,-0.5) rectangle (3.5,1.5);
  \fill[green!30!gray] (2.5,1.5) rectangle (3.5,3.5);
  \draw[xshift=-.5cm,yshift=-.5cm] (0,0) grid (4,4);
  \draw (-1,3) node {\tiny$0$} (0,4) node {\tiny$-1$};
  \draw (-1,2) node {\tiny$0$} (1,4) node {\tiny$-1$};
  \draw (-1,1) node {\tiny$\infty$} (2,4) node {\tiny$0$};
  \draw (-1,0) node {\tiny$\infty$} (3,4) node {\tiny$\infty$};
  \foreach\x/\y in {0/1, 0/3} \draw (\x,\y) node {$\bullet$};
  \foreach\x/\y in {1/0, 1/2} \draw (\x,\y) node {$\bullet$};
  \foreach\x/\y in {2/0, 2/1} \draw (\x,\y) node {$\bullet$};
  \foreach\x/\y in {3/2, 3/3} \draw (\x,\y) node {$\bullet$};
   \draw (1.5,-.75) node[below] {\vbox{\clap{$\mathstrut xy -1-y-x^2y-x^2y^2$}\kern0pt}};
\end{tikzpicture}

\end{center}
\end{Example}
\begin{Definition}
Let $p$ be a factor of $f_ng_d - g_nf_d$, and let $T$ be the corresponding invariant set. Let $s_1$ and $s_2$ be poles of $f$ and $g$, respectively, and let $T^{s_1,s_2}\subseteq T$ be such that $(i,j)\in T^{s_1,s_2}$ if and only if $(\alpha_i, \beta_j)$ is associated with $(s_1,s_2)$ in the sense of Proposition~\ref{prop:newt}. We refer to $T^{s_1,s_2}$ as the \textbf{component} of $T$ associated with $(s_1,s_2)$.
\end{Definition}

\begin{Example}\label{ex:comp}
Continuing with Example~\ref{ex:invSet}, the invariant set associated with $xy-1-y-x^2y-x^2y^2$ and $f_ng_d - g_nf_d$ has four non-empty components. The figure on the right above depicts its diagram again. Its rows and columns are not only numbered by the series roots of $f-t$ and $g-t$, respectively, but also labeled by the poles they encode. Its non-empty components are highlighted in color.
\end{Example}

We will see in Lemma~\ref{lemma:component} that the non-empty components of an invariant set associated with a factor of $f_ng_d - g_nf_d$ have an interpretation on the level of their leading parts. Before we present the lemma, we give another definition and another (simple) lemma that will turn out to be useful.
\begin{Definition}
Let $(s_1,s_2)$ be a pair of poles of $(f,g) \in\mathbb{K}(x)\times\mathbb{K}(y)$, and let $(\alpha,\beta)$ be a pair of roots of $f-t$ and $g-t$ in $\mathbb{K}\{\{t^{-1}\}\}$ associated with it. We say that $\omega\in\mathbb{Z}^2$ is associated with $(s_1,s_2)$ and $(f,g)$, if it is a positive multiple of $(\deg \alpha,\deg \beta)$.
\end{Definition}

\begin{Lemma}\label{lemma:set}
Let $S$ and $T$ be two sets, which are disjoint unions of sets $S_1$, $S_2$ and $T_1$, $T_2$, respectively, and let $\varphi: S \rightarrow T$ be a bijective map. If $\varphi$ restricts to injective maps between $S_1$, $T_1$ and $S_2$, $T_2$, respectively, then 
these restrictions are bijections.
\end{Lemma}

\begin{Lemma}\label{lemma:component}
Let $T$ be the invariant set of $p$ and $f_ng_d - g_nf_d$. Furthermore, let $(s_1,s_2)\in\{0,\infty\}^2$ be a pair of poles of $(f,g)$, and let $\omega\in\mathbb{Z}^2$ be associated with it. Then $T^{s_1,s_2}\neq \emptyset$ if and only if $\mathrm{lp}_\omega(p)$ is not a single term. If $T^{s1_,s_2} \neq \emptyset$, then the invariant set of $\mathrm{lp}_\omega(p)$ and $\mathrm{lp}_\omega(f) - \mathrm{lp}_\omega(g)$ can be identified with $T^{s_1,s_2}$.
\end{Lemma}
\begin{proof}
For notational convenience, given a series $\alpha\in\mathbb{K}\{\{t^{-1}\}\}$, we will denote its leading term by $\overline{\alpha}$.

Assume that $T^{s_1,s_2}\neq \emptyset$, and let $(i,j)\in T^{s_1,s_2}$. Then $p(\alpha_i,\beta_j) = 0$, and hence  
$\mathrm{lp}_\omega(p)(\overline{\alpha}_i,\overline{\beta}_j) = 0$ by extraction of terms of maximal degree. Since $\overline{\alpha}_i$, $\overline{\beta}_j\neq 0$, it follows that $\mathrm{lp}_\omega(p)$ involves at least two terms, and since it is a divisor of $\mathrm{lp}_\omega(f_ng_d - g_nf_d)$, so does the latter. Again, because of $\deg \alpha_i$, $\deg \beta_j \neq 0$, it follows that $\mathrm{lp}_\omega(f_ng_d - g_nf_d) = \mathrm{lp}_\omega(f_n)\mathrm{lp}_\omega(g_d) - \mathrm{lp}_\omega(g_n)\mathrm{lp}_\omega(f_d)$. By the construction of invariant sets, and by Proposition~\ref{prop:newt}, taking leading terms of series induces a map from $T^{s_1,s_2}$ to the invariant set of $\mathrm{lp}_\omega(p)$ and $\mathrm{lp}_\omega(f) - \mathrm{lp}_\omega(g)$. It is clearly injective since the series solutions of $f - t$ and $g-t$ associated with $s_1$ and $s_2$, respectively, can be distinguished by their leading terms. To see that it is also surjective, note that these observations do not only hold for $p$, but for any factor of $f_ng_d - g_nf_d$. In particular, it holds for the complementary factor of $p$ in $f_ng_d - g_nf_d$ and for $f_ng_d - g_nf_d$ itself. The set of pairs of series roots of $f-t$ and $g-t$ associated with $(s_1,s_2)$ and the set of pairs of series roots of $\mathrm{lp}_\omega(f) - t$ and $\mathrm{lp}_\omega(g) - t$ have equal size. Their cardinality is the product of $\mathrm{m}(s_1,f)$ and $\mathrm{m}(s_2,g)$. The map induced by taking leading terms is therefore not only injective. It is a surjection between these sets. The former set can be identified with the union of $T^{s_1,s_2}$ and the corresponding component of the invariant set associated with the complementary factor of $p$. The latter can be identified with the union of the invariant sets associated with their leading parts. It follows from Lemma ~\ref{lemma:set} that the restriction to $T^{s_1,s_2}$ is surjective too. The same argument proves the if-part of the statement.
\end{proof}

\begin{Example}\label{ex:comp1}
In Example~\ref{ex:comp} we observed that the invariant set $T$ associated with $p$ and $f_ng_d - g_nf_d$ partitions into four non-empty components. Two of them can be related with the leading parts of $p$ and $f_ng_d - g_nf_d$ with respect to $\omega_1 = (1,-2)$ and $\omega_2 = (-1,2)$. The other two with the leading parts of $p(x,-1+y)$ and $f_n(x)g_d(-1+y) - g_n(-1+y)f_d(x)$ with respect to $\omega_3 = (1,-1)$ and $\omega_4 = (-1,-1)$. The diagrams and the pairs of polynomials they are corresponding to are depicted below.
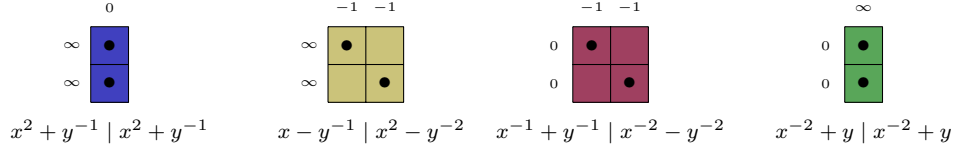
\begin{center}
\begin{tikzpicture}[scale=0.5]
  \fill[blue!50!gray] (-0.5,-0.5) rectangle (0.5,1.5);
  \draw[xshift=-.5cm,yshift=-.5cm] (0,0) grid (1,2);
  \draw (0,2) node {\tiny$0$};
  \draw (-1,1) node {\tiny$\infty$};
  \draw (-1,0) node {\tiny$\infty$};
  \foreach\x/\y in {0/0, 0/1} \draw (\x,\y) node {$\bullet$};
 \footnotesize
 \draw (0,-.75) node[below] {\vbox{\clap{$\mathstrut x^2+y^{-1} \mid x^2 + y^{-1}$}\kern0pt}};
\end{tikzpicture}
\qquad
\qquad
\hfil
\begin{tikzpicture}[scale=0.5]
 \fill[yellow!50!gray] (-0.5,-0.5) rectangle (1.5,1.5);
  \draw[xshift=-.5cm,yshift=-.5cm] (0,0) grid (2,2);
    \foreach\x/\y in {0/2, 1/2} \draw (\x,\y) node {\tiny$-1$};
  \draw (-1,1) node {\tiny$\infty$};
  \draw (-1,0) node {\tiny$\infty$};
  \foreach\x/\y in {1/0, 0/1} \draw (\x,\y) node {$\bullet$};
   \footnotesize
 \draw (0.65,-.75) node[below] {\vbox{\clap{$\mathstrut x-y^{-1} \mid x^2-y^{-2}$}\kern0pt}};
\end{tikzpicture}
\qquad
\quad
\hfil
\begin{tikzpicture}[scale=0.5]
  \fill[purple!50!gray]  (-0.5,-0.5) rectangle (1.5,1.5);
  \draw[xshift=-.5cm,yshift=-.5cm] (0,0) grid (2,2);
    \foreach\x/\y in {0/2, 1/2} \draw (\x,\y) node {\tiny$-1$};
  \draw (-1,1) node {\tiny$0$};
  \draw (-1,0) node {\tiny$0$};
  \foreach\x/\y in {1/0, 0/1} \draw (\x,\y) node {$\bullet$};
   \footnotesize
 \draw (0.5,-.75) node[below] {\vbox{\clap{$\mathstrut x^{-1}+y^{-1} \mid x^{-2}-y^{-2}$}\kern0pt}};
\end{tikzpicture}
\qquad
\qquad
\hfil
\begin{tikzpicture}[scale=0.5]
  \fill[green!30!gray] (-0.5,-0.5) rectangle (0.5,1.5);
  \draw[xshift=-.5cm,yshift=-.5cm] (0,0) grid (1,2);
  \draw (0,2) node {\tiny$\infty$};
  \draw (-1,1) node {\tiny$0$};
  \draw (-1,0) node {\tiny$0$};
  \foreach\x/\y in {0/0, 0/1} \draw (\x,\y) node {$\bullet$};
 \footnotesize
 \draw (0,-.75) node[below] {\vbox{\clap{$\mathstrut x^{-2} + y \mid x^{-2} + y$}\kern0pt}};
\end{tikzpicture}
\end{center}
\end{Example}

The diagrams of the components above have the same heights and lengths, respectively, when their vertical and horizontal sides are labeled by the same poles. This is not a coincidence.
\begin{Lemma}\label{lem:size}
Let $T$ be the invariant subset associated with $p$ and $f_ng_d - g_nf_d$, and let $s$, $s_1$, $s_2\in\mathbb{K}\cup\{\infty\}$ be such that $T^{s,s_1}$, $T^{s,s_2}\neq \emptyset$. Then
\begin{equation*}
\bigcup_i T^{s,s_1}_{*,i} = \bigcup_j T^{s,s_2}_{*,j}.
\end{equation*}
\end{Lemma}
\begin{proof}
Wlog we assume that $s, s_1, s_2 \in\{0,\infty\}$. Let $(i_k,j_k)\in T^{s,s_k}$, and define $\omega_k = (\deg \alpha_{i_k}, \deg \beta_{j_k})$, $k \in \{1,2\}$. By Lemma~\ref{lemma:component}, $T^{s,s_k}$ is the invariant set associated with $\mathrm{lp}_{\omega_k}(p)$ and $\mathrm{lp}_{\omega_k}(f) - \mathrm{lp}_{\omega_k}(g)$. Since $\deg \alpha_{i_k}$ is independent of $k$, so is $\mathrm{lp}_{\omega_k}(f)$. The statement now follows from the construction of invariant sets and Lemma~\ref{lem:card}.
\end{proof}

Lemma~\ref{lemma:factor} showed that there is a bijection between factors of $f_ng_d - g_nf_d$ and $G$-invariant subsets of $\mathbb{Z}_m\times\mathbb{Z}_n$. We next give a characterization of near-separated factors of $f_ng_d - g_nf_d$ in terms of properties of the invariant subsets associated with them.

\begin{Definition}
A subset $T\subseteq \mathbb{Z}_m\times\mathbb{Z}_n$ is called \textbf{separated} if
\begin{equation*}
\forall\ i,j\in\mathbb{Z}_m : (T_{i,*}\cap T_{j,*} = \emptyset) \text{ or } (T_{i,*} = T_{j,*}).
\end{equation*}
\end{Definition}

\begin{Lemma}\label{lem:combinatorial_separatedness}
 Let $p$ be a factor of $f_ng_d - g_nf_d$, and let $T\subseteq \mathbb{Z}_m\times \mathbb{Z}_n$ be the corresponding invariant set. Then $p$ is near-separated if and only if $T$ is separated.
\end{Lemma}
\begin{proof}
If $p = \tilde{f}_n\tilde{g}_d - \tilde{g}_n\tilde{f}_d$, then $(i,j)\in T$ if and only if $\tilde{f}(\alpha_i) = \tilde{g}(\beta_j)$. Hence, if $(i,j)$, $(i,j')$, $(i',j)\in T$ then $(i',j')\in T$. This shows the only-if part of the statement. 

Let us now assume that $T$ is separated, and let us show that the polynomial
\begin{equation*}
p(x,y) = a_s(x)y^s + a_{s-1}(x)y^{s-1} + \dots + a_0(x),
\end{equation*}
that corresponds to it is near-separated. By construction $a_i(\alpha_j)/a_s(\alpha_j)$ is, up to sign, the $(s-i)$-th elementary symmetric polynomial in $\{\beta_k : k\in T_{j,*}\}$ for each $0\leq  i < s$ and $0\leq j <m$. By assumption, $T$ is separated. If $k\in T_{j,*} \cap T_{j',*}$, then $T_{j,*} = T_{j',*}$, and therefore $a_i(\alpha_j)/a_s(\alpha_j) = a_i(\alpha_{j'})/a_s(\alpha_{j'})$ by construction. We will crucially make use of this observation in a moment.

Let us assume for the moment that $\val a_s > 0$. If $i_0$ is such that $\val a_{i_0}=0$, then there are $c_i\in\mathbb{K}$ such that $\val(a_i - c_i a_{i_0}) > 0$. The number of non-zero roots of
\begin{equation*}
\frac{a_i(x) - c_i a_{i_0}(x)}{a_s(x)} - \frac{a_i(\alpha_0) - c_i a_{i_0}(\alpha_0)}{a_s(\alpha_0)}
\end{equation*}
is at most
\begin{equation*}
\max\{ \deg (a_i - c_i a_{i_0}), \deg a_s\} - \min\{ \val (a_i - c_i a_{i_0}), \val a_s\},
\end{equation*}
and therefore smaller than $\deg_x p$. If $j\in T_{0,*}$, then for each $k\in T_{*,j}$ the series $\alpha_k$ is a root. Since these roots are non-zero and pairwise distinct, and because there are $\deg_x p$ of them (see the proof of Lemma~\ref{lemma:factor}), the rational function is identically zero. Hence there are $d_i\in\mathbb{K}$ such that
\begin{equation*}
a_i(x) = c_i a_{i_0}(x) + d_i a_s(x).
\end{equation*}
Consequently,
\begin{equation*}
p(x,y) = \sum_{i=0}^s a_i(x) y^i = a_{i_0}(x) \sum_{i=0}^s c_i y^i + a_s(x) \sum_{i=0}^s d_i y^i,
\end{equation*}
that is, $p$ is near-separated. If $\val a_s = 0$ and $a_s$ is not just a single term, then it has a root $c\in\mathbb{K}$. The leading coefficient of $p(x+c,y)$ with respect to $y$ has positive valuation, and we can argue as before to show that $p(x+c,y)$, and hence also $p$, is near-separated. If $\val a_s = 0$ and $a_s$ is just a single term, then $a_s$ is a constant, and $\deg a_s < \deg_x p$. Choosing $i_0$ such that $\deg a_{i_0} = \deg_x p$ and $c_i\in\mathbb{K}$ such that $\deg(a_i - c_i a_{i_0}) < \deg_x p$, we can argue as before to show that $a_i$ is a linear combination of $a_{i_0}$ and $a_s$ to conclude that $p$ is near-separated.
\end{proof}

We present another definition and another lemma before we come to the main theorem. The proof of Lemma~\ref{lem:closure_invariance} is taken literally from~\cite[Lemma~3.13.]{buchacher2020separating} and included here for convenience of the reader.
\begin{Definition}
Let $T$ be an invariant subset of $\mathbb{Z}_m\times \mathbb{Z}_n$. The \textbf{separated closure} of $T$ is
\begin{equation*}
T^{\mathrm{sep}} := \bigcap_{\substack{S\supseteq T\\ S \text{ sep}}} S.
\end{equation*}
\end{Definition}

\begin{Lemma}\label{lem:closure_invariance}
    Let $T \subseteq \mathbb{Z}_m \times \mathbb{Z}_n$ be invariant with respect to $G \subseteq \mathrm{S}_m \times \mathrm{S}_n$.
    Then $T^{\sep}$ is also $G$-invariant.
\end{Lemma}

\begin{proof}
    Let $\pi = (\sigma, \tau) \in \mathrm{S}_m \times \mathrm{S}_n$, and let
    $S \subseteq \mathbb{Z}_m \times \mathbb{Z}_n$ be a separated set.
    Since $\pi(S)_{i, \ast} = \tau(S_{\sigma(i), \ast})$,
    we find that $\pi(S)$ is separated as well.

    Assume that $T^{\sep}$ is not $G$-invariant, that is, there exists a $\pi \in G$
    such that $\pi(T^{\sep}) \neq T^{\sep}$.
    As we have shown, $\pi(T^{\sep})$ is separated, hence so is $S := T^{\sep} \cap \pi(T^{\sep})$.
    Observe that, since $\pi(T^{\sep}) \neq T^{\sep}$, $S \subsetneq T^{\sep}$.
    Since $T$ is $G$-invariant, $T \subseteq \pi(T^{\sep})$, so $T \subseteq S$.
    This contradicts the minimality of $T^{\sep}$.
\end{proof}

We are finally able to prove the correctness Algorithm~\ref{alg:2}. 

\begin{Theorem}~\label{theorem:main}
Let $p\in\mathbb{K}[x,y]\setminus\left(\mathbb{K}[x]\cup \mathbb{K}[y]\right)$ be irreducible, and assume that $\mathrm{F}(p)$ is non-trivial, that is, not isomorphic to $\mathbb{K}$. Then Algorithm~\ref{alg:2} terminates on input $p$, and outputs a generator of $\mathrm{F}(p)$.
\end{Theorem}
\begin{proof}
Let $f_ng_d-g_nf_d$ be the minimal near-separated multiple of $p$ such that $\deg f_n > \deg f_d$, and let $T$ be the invariant set associated with them. By Lemma~\ref{lem:combinatorial_separatedness}, $T^{\mathrm{sep}}$, the separated closure of $T$, is all of $\mathbb{Z}_m\times\mathbb{Z}_n$.

In order to prove the correctness of Algorithm~\ref{alg:2}, we first draw our attention to Algorithm~\ref{alg:1}, which is used as a subroutine, and show that it does not miss any poles of $f$ and $g$. We will prove that if Algorithm~\ref{alg:1} finds a pole $s_1$ of $f$, and if $s_2$ is a pole of $g$ such that $T^{s_1,s_2} \neq \emptyset$, then it also finds $s_2$. This will also hold with the roles of $s_1$ and $s_2$ interchanged. This will then imply that, if Algorithm~\ref{alg:1} did find some but not all poles of $f$ and $g$, then $T$ were the union of two-nonempty sets $T_0$ and $T_1$, which, after a permutation of its rows and columns, can be assumed to be subsets of $\{0,1,\dots,m_0\}\times \{0,1,\dots,n_0\}$ and $\{m_0+1,\dots,m-1\}\times \{n_0+1,\dots,n-1\}$, respectively. This would then imply that $T^\mathrm{sep}$ equals $T_0^\mathrm{sep} \cup T_1^\mathrm{sep}$, a proper subset of $\mathbb{Z}_m\times\mathbb{Z}_n$, and thereby contradict the assumption that $f_ng_d - g_nf_d$ is the minimal separated multiple of $p$. 

Let $s_1$ be a pole of $f$ that has been determined by Algorithm~\ref{alg:1}, and let $s_2$ be a pole of $g$ such that $T^{s_1,s_2} \neq \emptyset$. We can assume that $s_1$ is either $0$ or $\infty$. We assume that it is $0$. The other case is treated analogously.  If $(i,j)\in T^{s_1,s_2}$, then $p(\alpha_i,\beta_j) = 0$, and therefore $\mathrm{lp}_\omega(p)(\overline{\alpha}_i,\overline{\beta}_j) = 0$ for $\omega = (\deg \alpha_i, \deg \beta_j)$. Since $\alpha_i$ and $\beta_j$ are different from zero, the leading part $\mathrm{lp}_\omega(p)$ involves at least two terms. Therefore, $\omega$ is an outward pointing normal of an edge of the Newton polygon of $p$. 
If $s_2 = \infty$, then $\mathrm{sgn}(w) = (-1,1)$, and so $\deg p(0,y) < \deg_y p$. If $s_2 \in\mathbb{K}$, then $\mathrm{sgn}(w)$ equals either $(-1,0)$ or $(-1,-1)$, depending on whether $s_2$ is zero or not. In any case, $s_2$ is a root of $p(0,y)$. Altogether, we see that Algorithm~\ref{alg:1} succeeds in finding $s_2$. 

We show that the multiplicities computed by Algorithm~\ref{alg:2} are indeed the multiplicities of the poles of $f$ and $g$.
Let $T^{s_1,s_2}$ be a non-empty component of $T$, and assume it corresponds to the invariant set associated with $\mathrm{lp}_\omega(p)$ and $\mathrm{lp}_\omega(f_n) \mathrm{lp}_\omega(g_d) - \mathrm{lp}_\omega(g_n)\mathrm{lp}_\omega(f_d)$ as explained in Lemma~\ref{lemma:component}. If $(f_\omega,  g_\omega)$ is a generator of $\mathrm{F}(\mathrm{lp}(p))$, then there is an integer $k$ such that $(\mathrm{lp}_\omega(f), \mathrm{lp}_\omega(g)) = (f_\omega^k, g_\omega^k)$. So the invariant set~$T_k^{s_1,s_2}$ associated with $\mathrm{lp}_\omega(p)$ and $f_\omega^k - g_\omega^k$ can be identified with $T^{s_1,s_2}$. If for each pair $(s_1,s_2)$ of poles, $k \equiv k(s_1,s_2)$ is chosen such that $T_k^{s_1,s_2}$ can be identified with $T^{s_1,s_2}$, then Lemma~\ref{lem:size} shows that the diagrams associated with $T_k^{s_1,s_2}$ need to be compatible in terms of size. This compatibility does not uniquely determine the $k$'s but gives rise to a $1$-parameter family of sets $T_k$ only one of which can be identified with $T$.
Again, since $f_ng_d - g_nf_d$ is the minimal near-separated multiple of~$p$, the separated closure of~$T$ equals $\mathbb{Z}_m\times \mathbb{Z}_n$. We claim that the only $k$ for which the separated closure of $T_k \subseteq \mathbb{Z}_{km'}\times\mathbb{Z}_{kn'}$ equals $\mathbb{Z}_{km'}\times\mathbb{Z}_{kn'}$ is $1$.
To see this, we will compare $T_1^{s_1, s_2}$ to $T^{s_1,s_2}_k$ and $T_1$ to $T_k$. Let
\begin{equation*}
\alpha_{k,i} = \mathrm{exp}\left(\frac{2\pi\mathrm{i} i}{km'}\right)a^{-\frac{1}{m'}} t^\frac{1}{km'}
\quad \text{and} \quad
\beta_{k,j} = \mathrm{exp}\left(\frac{2\pi \mathrm{i} j}{kn'}\right)b^{-\frac{1}{n'}} t^\frac{1}{kn'}
\end{equation*}
be the solutions of the auxiliary equations~$f_\omega^k - t = 0$ and $g_\omega^k - t =0$. Since
\begin{equation*}
t^\frac{1}{km'n'} \quad \mapsto \quad \mathrm{exp}\left(\frac{2\pi\mathrm{i}}{km'n'}\right) t^\frac{1}{km'n'}
\end{equation*}
is an element of $\mathrm{Gal}(\mathbb{K}(t^\frac{1}{km'n'})/\mathbb{K}(t))$, we see that $(\alpha_{k,i},\beta_{k,j})$ is a root of $\mathrm{lp}_\omega(p)$ if and only if $(\alpha_{k,i+1 \bmod km'},\beta_{k,j+1 \bmod kn'})$ is one. Furthermore, $(\alpha_{k,ki},\beta_{k,kj})$ is a root of $\mathrm{lp}_\omega(p)$ if and only if $(\alpha_{1,i},\beta_{1,j})$ is one, since
\begin{equation*}
\alpha_{k,ki}(t) = \alpha_{1,i}(t^{1/k}) \quad \text{and} \quad \beta_{k,kj}(t) = \beta_{1,j}(t^{1/k}).
\end{equation*}
From these two observations it is not too difficult to deduce that the permutation of $\mathbb{Z}_{km'} \times \mathbb{Z}_{kn'}$ given by
\begin{equation*}
(u_1 k + v_1, u_2 k + v_2) \mapsto (v_1 m' + u_1, v_2 n' +u_2),
\end{equation*}
where $u_1 \in \{0,\dots,m'-1\}$, $u_2\in\{0,\dots,n'-1\}$ and $v_1,v_2\in\{0,\dots,k-1\}$ permutes the rows and columns of $T_k^{s_1,s_2}$ such that the associated diagram is of block diagonal form with each block equal to the diagram associated with $T_1^{s_1,s_2}$. These permutations indexed by the pairs $(s_1,s_2)$ of poles of $(f,g)$ make up a permutation of the rows and columns of $T_k$ such that the diagram of the component associated with $(s_1,s_2)$ is of block diagonal form as above. These blocks again can be permuted such that the corresponding diagram is of block diagonal form with each block equal to the diagram associated with $T_1$. Let us write $\diag_k(T_1)$ for it, or more generally for a diagram in block diagonal form, consisting of $k$ blocks of a diagram $T_1$. Then the diagram of the separated closure is $\diag_k(T^{\mathrm{sep}})$. Obviously, if $k>1$, the separated closure of $T_k$ is a proper subset of $\mathbb{Z}_{km'}\times\mathbb{Z}_{kn'}$. Hence $T = T_1$, and the proof is finished.
 \end{proof}

\begin{Example}
In Example~\ref{ex:comp1} we computed the invariant sets associated with $x^2+y^{-1}$, $x^{-2}+y$, $x-y^{-1}$ and $x^{-1}+y^{-1}$ as factors of $x^2+y^{-1}$, $x^{-2}+y$, $x^2-y^{-2}$ and $x^{-2}-y^{-2}$, respectively, and observed that they are the components of the invariant set associated with $p=xy-1-y-x^2y-x^2y^2$ and $f_ng_d-g_nf_d = (1-x)^2(1+x+x^2)y(1+y)^2+(1+y+y^2)^2x^2$. The invariant sets associated with $x^2+y^{-1}$, $x^{-2}+y$, $x-y^{-1}$ and $x^{-1}+y^{-1}$ as factors of $x^4+y^{-2}$, $x^{-4}+y^2$, $x^4-y^{-4}$ and $x^{-4}-y^{-4}$, respectively, are depicted below.
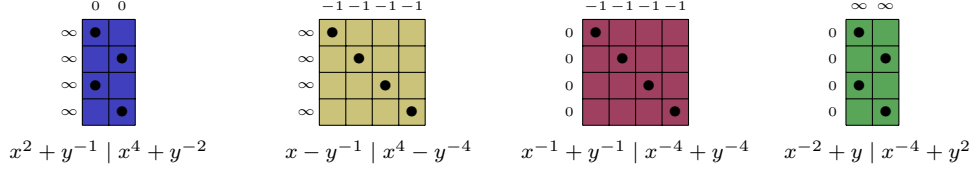
\begin{center}
\begin{tikzpicture}[scale=0.35]
  \fill[blue!50!gray] (-0.5,-0.5) rectangle (1.5,3.5);
  \draw[xshift=-.5cm,yshift=-.5cm] (0,0) grid (2,4);
  \draw (0,4) node {\tiny$0$};
  \draw (1,4) node {\tiny$0$};
  \draw (-1,3) node {\tiny$\infty$};
  \draw (-1,2) node {\tiny$\infty$};
  \draw (-1,1) node {\tiny$\infty$};
  \draw (-1,0) node {\tiny$\infty$};
  \foreach\x/\y in {0/3, 0/1, 1/2, 1/0} \draw (\x,\y) node {$\bullet$};
   \footnotesize
 \draw (0.5,-.75) node[below] {\vbox{\clap{$\mathstrut  x^2+y^{-1}  \mid x^4+ y^{-2}$}\kern0pt}};
\end{tikzpicture}
\qquad
\qquad
\hfil
\begin{tikzpicture}[scale=0.35]
 \fill[yellow!50!gray] (-0.5,-0.5) rectangle (3.5,3.5);
  \draw[xshift=-.5cm,yshift=-.5cm] (0,0) grid (4,4);
    \foreach\x/\y in {0/4, 1/4, 2/4, 3/4} \draw (\x,\y) node {\tiny$-1$};
  \draw (-1,3) node {\tiny$\infty$};
  \draw (-1,2) node {\tiny$\infty$};
  \draw (-1,1) node {\tiny$\infty$};
  \draw (-1,0) node {\tiny$\infty$};
  \foreach\x/\y in {0/3, 1/2, 2/1, 3/0} \draw (\x,\y) node {$\bullet$};
   \footnotesize
 \draw (1.75,-.75) node[below] {\vbox{\clap{$\mathstrut x-y^{-1} \mid x^4-y^{-4}$}\kern0pt}};
\end{tikzpicture}
\qquad
\quad
\hfil
\begin{tikzpicture}[scale=0.35]
 \fill[purple!50!gray] (-0.5,-0.5) rectangle (3.5,3.5);
  \draw[xshift=-.5cm,yshift=-.5cm] (0,0) grid (4,4);
    \foreach\x/\y in {0/4, 1/4, 2/4, 3/4} \draw (\x,\y) node {\tiny$-1$};
  \draw (-1,3) node {\tiny$0$};
  \draw (-1,2) node {\tiny$0$};
  \draw (-1,1) node {\tiny$0$};
  \draw (-1,0) node {\tiny$0$};
  \foreach\x/\y in {0/3, 1/2, 2/1, 3/0} \draw (\x,\y) node {$\bullet$};
   \footnotesize
 \draw (1.5,-.75) node[below] {\vbox{\clap{$\mathstrut x^{-1}+y^{-1} \mid x^{-4}+y^{-4}$}\kern0pt}};
\end{tikzpicture}
\qquad
\quad
\hfil
\begin{tikzpicture}[scale=0.35]
  \fill[green!30!gray] (-0.5,-0.5) rectangle (1.5,3.5);
  \draw[xshift=-.5cm,yshift=-.5cm] (0,0) grid (2,4);
  \draw (0,4) node {\tiny$\infty$};
  \draw (1,4) node {\tiny$\infty$};
  \draw (-1,3) node {\tiny$0$};
  \draw (-1,2) node {\tiny$0$};
  \draw (-1,1) node {\tiny$0$};
  \draw (-1,0) node {\tiny$0$};
  \foreach\x/\y in {0/3, 0/1, 1/2, 1/0} \draw (\x,\y) node {$\bullet$};
   \footnotesize
 \draw (0.7,-.75) node[below] {\vbox{\clap{$\mathstrut  x^{-2} + y  \mid x^{-4}+ y^2$}\kern0pt}};
\end{tikzpicture}
\end{center}
They make up the invariant set $T_2\subseteq\mathbb{Z}_{8}\times\mathbb{Z}_{8}$ associated with $p$ and $f^2 - g^2$ (see the diagram to the left below). A permutation of the rows and columns of $T_2$ results in a diagram of block diagonal form with its two blocks corresponding to $T$ (see the diagram to the right below).

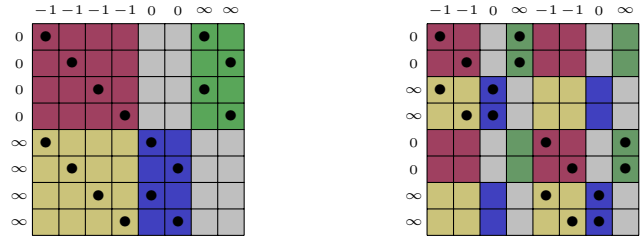
\begin{center}
\begin{tikzpicture}[scale=0.35]
 \fill[yellow!50!gray] (-0.5,-0.5) rectangle (3.5,3.5);
 \fill[purple!50!gray] (-0.5,3.5) rectangle (3.5,7.5);
 \fill[blue!50!gray] (3.5,-0.5) rectangle (5.5,3.5);
 \fill[lightgray] (5.5,-0.5) rectangle (7.5,3.5);
 \fill[lightgray] (3.5,3.5) rectangle (5.5,7.5);
 \fill[green!30!gray] (5.5,3.5) rectangle (7.5,7.5);
  \draw[xshift=-.5cm,yshift=-.5cm] (0,0) grid (8,8);
  \draw (-1,7) node {\tiny$0$} (3,8) node {\tiny$-1$};
  \draw (-1,6) node {\tiny$0$} (2,8) node {\tiny$-1$};
  \draw (-1,5) node {\tiny$0$} (1,8) node {\tiny$-1$};
  \draw (-1,4) node {\tiny$0$} (0,8) node {\tiny$-1$};
  \foreach\x/\y in {4/8, 5/8} \draw (\x,\y) node {\tiny$0$};
  \foreach\x/\y in {6/8, 7/8} \draw (\x,\y) node {\tiny$\infty$};
  \foreach\x/\y in {-1/0, -1/1, -1/2, -1/3} \draw (\x,\y) node {\tiny$\infty$};
  \foreach\x/\y in {3/0, 2/1, 1/2, 0/3} \draw (\x,\y) node {$\bullet$};
  \foreach\x/\y in {3/4, 2/5, 1/6, 0/7} \draw (\x,\y) node {$\bullet$};
  \foreach\x/\y in {4/3, 4/1, 5/2, 5/0} \draw (\x,\y) node {$\bullet$};
  \foreach\x/\y in {6/7, 6/5, 7/6, 7/4} \draw (\x,\y) node {$\bullet$};
\end{tikzpicture}
\hfil
\begin{tikzpicture}[scale=0.35]
 \fill[lightgray] (-0.5,-0.5) rectangle (3.5,3.5);
 \fill[lightgray] (3.5,3.5) rectangle (7.5,7.5);
 \fill[purple!50!gray] (-0.5,5.5) rectangle (1.5,7.5);
 \fill[lightgray] (1.5,5.5) rectangle (2.5,7.5);
 \fill[green!20!gray] (2.5,5.5) rectangle (3.5,7.5);
 \fill[yellow!50!gray] (-0.5,3.5) rectangle (1.5,5.5);
 \fill[blue!50!gray] (1.5,3.5) rectangle (2.5,5.5);
 \fill[lightgray] (2.5,3.5) rectangle (3.5,5.5);

 \fill[purple!50!gray] (3.5,1.5) rectangle (5.5,3.5);
 \fill[lightgray] (5.5,1.5) rectangle (6.5,3.5);
 \fill[green!20!gray] (6.5,1.5) rectangle (7.5,3.5);
 \fill[yellow!50!gray] (3.5,-0.5) rectangle (5.5,1.5);
 \fill[blue!50!gray] (5.5,-0.5) rectangle (7.5,1.5);
 \fill[lightgray] (6.5,-0.5) rectangle (7.5,1.5);

  \fill[purple!50!gray] (-0.5,1.5) rectangle (1.5,3.5);
 \fill[lightgray] (1.5,1.5) rectangle (2.5,3.5);
 \fill[green!20!gray] (2.5,1.5) rectangle (3.5,3.5);
 \fill[yellow!50!gray] (-0.5,-0.5) rectangle (1.5,1.5);
 \fill[blue!50!gray] (1.5,-0.5) rectangle (3.5,1.5);
 \fill[lightgray] (2.5,-0.5) rectangle (3.5,1.5);

  \fill[purple!50!gray] (3.5,5.5) rectangle (5.5,7.5);
 \fill[lightgray] (5.5,5.5) rectangle (6.5,7.5);
 \fill[green!20!gray] (6.5,5.5) rectangle (7.5,7.5);
 \fill[yellow!50!gray] (3.5,3.5) rectangle (5.5,5.5);
 \fill[blue!50!gray] (5.5,3.5) rectangle (7.5,5.5);
 \fill[lightgray] (6.5,3.5) rectangle (7.5,5.5);

  \draw[xshift=-.5cm,yshift=-.5cm] (0,0) grid (8,8);
  \draw (-1,7) node {\tiny$0$} (1,8) node {\tiny$-1$};
  \draw (-1,6) node {\tiny$0$} (0,8) node {\tiny$-1$};
  \draw (-1,5) node {\tiny$\infty$} (2,8) node {\tiny$0$};
  \draw (-1,4) node {\tiny$\infty$} (3,8) node {\tiny$\infty$};
  \draw (-1,3) node {\tiny$0$} (4,8) node {\tiny$-1$};
  \draw (-1,2) node {\tiny$0$} (5,8) node {\tiny$-1$};
  \draw (-1,1) node {\tiny$\infty$} (6,8) node {\tiny$0$};
  \draw (-1,0) node {\tiny$\infty$} (7,8) node {\tiny$\infty$};
  \foreach\x/\y in {0/7, 3/7, 1/6, 3/6, 0/5, 2/5, 1/4, 2/4, 4/3, 7/3, 5/2, 7/2, 4/1, 6/1, 5/0, 6/0} \draw (\x,\y) node {$\bullet$};
\end{tikzpicture}
\end{center}

\end{Example}

\section{Open problem}\label{section:questions}

We have discussed a semi-algorithm that takes as input an irreducible polynomial $p\in\mathbb{K}[x,y]\setminus\left(\mathbb{K}[x]\cup \mathbb{K}[y]\right)$ and outputs a generator of $\mathrm{F}(p)$ whenever it terminates. We have observed that it does terminate, when $p$ is near-separable, and we have shown that it may not if $p$ is not near-separable. 

We have seen several necessary conditions for a polynomial to be near-separable: the sign vectors of the outward pointing normals of any two distinct edges of its Newton polygon need to be different; the orbit of $\infty$ (and of any other point) is necessarily finite; its leading parts have to be near-separable; and the linear equations for the coefficients of the ansatz of a near-separated multiple need to have a non-trivial solution. The only condition that is difficult to verify is the (in)finiteness of the orbit of a point. Though~\cite[Remark~5.1]{hardouin2021differentially} provides an answer how this can be done it certain situations, it remains an open question how it can be addressed in general.

\begin{Problem}
Given $p\in\mathbb{K}[x,y]$, decide whether the orbit $\mathcal{O}$ of $\infty$ is (in)finite.
\end{Problem}

\section{Acknowledgements}

Thanks go to the Johann Radon Institute for Computational and Applied Mathematics of the Austrian Academy of Sciences at which the author was employed while part of this work was carried out. Thanks also go to the Johannes Kepler University Linz which supported this work with the grant LIT-2022-11-YOU-214.

\bibliographystyle{plain}
\bibliography{nearSeparationCleanUp}

\begin{thebibliography}{10}

\bibitem{alonso1995rational}
Cesar Alonso, Jaime Gutierrez, and Tomas Recio.
\newblock A rational function decomposition algorithm by near-separated
  polynomials.
\newblock {\em Journal of Symbolic Computation}, 19(6):527--544, 1995.

\bibitem{alonso1997note}
Cesar Alonso, Jaime Gutierrez, and Tomas Recio.
\newblock A note on separated factors of separated polynomials.
\newblock {\em Journal of Pure and Applied Algebra}, 121(3):217--222, 1997.

\bibitem{ash1986intersections}
Christopher~J Ash and John~W Rosenthal.
\newblock Intersections of algebraically closed fields.
\newblock {\em Annals of pure and applied logic}, 30(2):103--119, 1986.

\bibitem{bernardi2020counting}
Olivier Bernardi, Mireille Bousquet-M{\'e}lou, and Kilian Raschel.
\newblock Counting quadrant walks via {T}utte's invariant method.
\newblock {\em Discrete Mathematics \& Theoretical Computer Science}, 2020.

\bibitem{bilu2000diophantine}
Yuri Bilu and Robert Tichy.
\newblock The diophantine equation $f(x) = g(y)$.
\newblock {\em Acta Arithmetica}, 95(3):261--288, 2000.

\bibitem{binder2009algorithms}
Anna~Katharina Binder.
\newblock {\em Algorithms for Fields and an Application to a Problem in
  Computer Vision}.
\newblock PhD thesis, Technische Universit{\"a}t M{\"u}nchen, 2009.

\bibitem{binder1996fast}
Franz Binder.
\newblock Fast computations in the lattice of polynomial rational function
  fields.
\newblock In {\em Proceedings of the 1996 {I}nternational {S}ymposium on
  Symbolic and {A}lgebraic {C}omputation}, pages 43--48, 1996.

\bibitem{bonnet2024galoisian}
Pierre Bonnet and Charlotte Hardouin.
\newblock Galoisian structure of large steps walks in the quadrant.
\newblock {\em arXiv preprint arXiv:2405.03508}, 2024.

\bibitem{bousquet2023enumeration}
Mireille Bousquet-M{\'e}lou.
\newblock Enumeration of three-quadrant walks via invariants: some diagonally
  symmetric models.
\newblock {\em Canadian Journal of Mathematics}, 75(5):1566--1632, 2023.

\bibitem{bousquet2006polynomial}
Mireille Bousquet-M{\'e}lou and Arnaud Jehanne.
\newblock Polynomial equations with one catalytic variable, algebraic series
  and map enumeration.
\newblock {\em Journal of Combinatorial Theory, Series B}, 96(5):623--672,
  2006.

\bibitem{bousquet2010walks}
Mireille Bousquet-M{\'e}lou and Marni Mishna.
\newblock Walks with small steps in the quarter plane.
\newblock {\em Contemp. Math}, 520:1--40, 2010.

\bibitem{buchacher2024separated}
Manfred Buchacher.
\newblock Separated variables on plane algebraic curves.
\newblock {\em arXiv preprint arXiv:2411.08584}, 2024.

\bibitem{buchacher2020separating}
Manfred Buchacher, Manuel Kauers, and Gleb Pogudin.
\newblock Separating variables in bivariate polynomial ideals.
\newblock In {\em Proceedings of the 45th International Symposium on Symbolic
  and Algebraic Computation}, pages 54--61, 2020.

\bibitem{buchberger2006bruno}
Bruno Buchberger.
\newblock Bruno buchberger’s {P}h{D} thesis 1965: An algorithm for finding
  the basis elements of the residue class ring of a zero dimensional polynomial
  ideal.
\newblock {\em Journal of symbolic computation}, 41(3-4):475--511, 2006.

\bibitem{dreyfus2018nature}
Thomas Dreyfus, Charlotte Hardouin, Julien Roques, and Michael~F Singer.
\newblock On the nature of the generating series of walks in the quarter plane.
\newblock {\em Inventiones mathematicae}, 213(1):139--203, 2018.

\bibitem{dreyfus2020walks}
Thomas Dreyfus, Charlotte Hardouin, Julien Roques, and Michael~F Singer.
\newblock Walks in the quarter plane: genus zero case.
\newblock {\em Journal of Combinatorial Theory, Series A}, 174:105251, 2020.

\bibitem{dreyfus2024enumeration}
Thomas Dreyfus, Andrew~Elvey Price, and Kilian Raschel.
\newblock Enumeration of weighted quadrant walks: criteria for algebraicity and
  {D}-finiteness.
\newblock {\em arXiv preprint arXiv:2409.12806}, 2024.

\bibitem{fried1973field}
Michael Fried.
\newblock The field of definition of function fields and a problem in the
  reducibility of polynomials in two variables.
\newblock {\em Illinois Journal of Mathematics}, 17(1):128--146, 1973.

\bibitem{fried1973theorem}
Michael Fried.
\newblock On a theorem of {R}itt and related {D}iophantine problems.
\newblock 1973.

\bibitem{fried1974arithmetical}
Michael Fried.
\newblock Arithmetical properties of function fields. {II}. the generalized
  {S}chur problem.
\newblock {\em Acta Arith}, 25(3):225--258, 1974.

\bibitem{fried1978poncelet}
Michael Fried.
\newblock Poncelet correspondences: Finite correspondences; {R}itt's theorem;
  and the {G}riffiths-{H}arris configuration for quadrics.
\newblock {\em Journal of Algebra}, 54(2):467--493, 1978.

\bibitem{fried1980exposition}
Michael Fried.
\newblock Exposition on an arithmetic-group theoretic connection via
  riemann’s existence theorem.
\newblock In {\em Proceedings of Symposia in Pure Math: Santa Cruz Conference
  on Finite Groups, AMS Publications}, volume~37, pages 571--601, 1980.

\bibitem{fried1987irreducibility}
Michael Fried.
\newblock Irreducibility results for separated variables equations.
\newblock {\em Journal of Pure and Applied Algebra}, 48(1-2):9--22, 1987.

\bibitem{fried1999variables}
Michael Fried.
\newblock Variables separated polynomials, the genus 0 problem and moduli
  spaces.
\newblock {\em Number theory in progress}, 1:169--228, 1999.

\bibitem{geyer1994irreducibility}
Wulf-Dieter Geyer.
\newblock On the irreducibility of sums of rational functions with separated
  variables.
\newblock {\em Israel Journal of Mathematics}, 85:135--168, 1994.

\bibitem{hardouin2021differentially}
Charlotte Hardouin and Michael~F Singer.
\newblock On differentially algebraic generating series for walks in the
  quarter plane.
\newblock {\em Selecta Mathematica}, 27(5):89, 2021.

\bibitem{kemper1993algorithm}
Gregor Kemper.
\newblock {\em An algorithm to determine properties of field extensions lying
  over a ground field}.
\newblock IWR, 1993.

\bibitem{kreso2015functional}
Dijana Kreso and {Robert F.} Tichy.
\newblock Functional composition of polynomials: indecomposability,
  {D}iophantine equations and lacunary polynomials.
\newblock {\em Grazer mathematische Berichte}, 363:143--170, 2015.

\bibitem{luroth1875beweis}
Jakob L{\"u}roth.
\newblock Beweis eines {S}atzes {\"u}ber rationale {C}urven.
\newblock {\em Mathematische Annalen}, 9(2):163--165, 1875.

\bibitem{maclagan2015introduction}
Diane Maclagan and Bernd Sturmfels.
\newblock {\em Introduction to tropical geometry}, volume 161.
\newblock American Mathematical Soc., 2015.

\bibitem{mishna2009classifying}
Marni Mishna.
\newblock Classifying lattice walks restricted to the quarter plane.
\newblock {\em Journal of Combinatorial Theory, Series A}, 116(2):460--477,
  2009.

\bibitem{muller1998algorithmic}
J{\"o}rn M{\"u}ller-Quade, Harald Aagedal, Th~Beth, and Michael Schmid.
\newblock Algorithmic design of diffractive optical systems for information
  processing.
\newblock {\em Physica D: Nonlinear Phenomena}, 120(1-2):196--205, 1998.

\bibitem{muller1999basic}
J{\"o}rn M{\"u}ller-Quade and Rainer Steinwandt.
\newblock Basic algorithms for rational function fields.
\newblock {\em Journal of Symbolic Computation}, 27(2):143--170, 1999.

\bibitem{muller2000grobner}
J{\"o}rn M{\"u}ller-Quade and Rainer Steinwandt.
\newblock Gr{\"o}bner bases applied to finitely generated field extensions.
\newblock {\em Journal of Symbolic Computation}, 30(4):469--490, 2000.

\bibitem{ovchinnikov2021computing}
Alexey Ovchinnikov, Anand Pillay, Gleb Pogudin, and Thomas Scanlon.
\newblock Computing all identifiable functions of parameters for ode models.
\newblock {\em Systems \& Control Letters}, 157:105030, 2021.

\bibitem{rotman2015advanced}
Joseph~J Rotman.
\newblock Advanced modern algebra. {P}art 1, {V}olume 165 of {G}raduate
  {S}tudies in {M}athematics.
\newblock {\em American Mathematical Society, Providence, RI,}, 7, 2015.

\bibitem{schicho1995note}
Josef Schicho.
\newblock A note on a theorem of {F}ried and {M}ac{R}ae.
\newblock {\em Archiv der Mathematik}, 65(3):239--243, 1995.

\bibitem{steinwandt2000freeness}
Rainer Steinwandt and J{\"o}rn M{\"u}ller-Quade.
\newblock Freeness, linear disjointness, and implicitization—a classical
  approach.
\newblock {\em Beitr{\"a}ge Algebra Geom}, 41(1):57--66, 2000.

\bibitem{sweedler1993using}
Moss Sweedler.
\newblock Using {G}r{\"o}bner bases to determine the algebraic and
  transcendental nature of field extensions: return of the killer tag
  variables.
\newblock In {\em Applied Algebra, Algebraic Algorithms and Error-Correcting
  Codes: 10th International Symposium, AAECC-10 San Juan de Puerto Rico, Puerto
  Rico, May 10--14, 1993 Proceedings 10}, pages 66--75. Springer, 1993.

\bibitem{walker1950algebraic}
Robert~John Walker.
\newblock {\em Algebraic curves}, volume~58.
\newblock Springer, 1950.

\end{thebibliography}

\end{document}